\documentclass[11pt]{article}
\usepackage{tikz}
\usepackage[a4paper,margin=1in]{geometry}
\usepackage{amsmath,amssymb,amsthm,mathtools}
\usepackage{graphicx}
\usepackage{aliascnt}
\usepackage{enumitem}
\setlist[enumerate]{
  label=\textup{(\roman*)},
  labelindent=0pt,
  leftmargin=*,
  align=left
}
\usepackage{microtype}
\usepackage{xcolor}
\usepackage[colorlinks=true,linkcolor=blue!65!black,citecolor=blue!65!black,urlcolor=blue!65!black]{hyperref}
\usepackage[nameinlink,noabbrev]{cleveref}

\numberwithin{equation}{section}

\newtheorem{theorem}{Theorem}[section]

\newaliascnt{proposition}{theorem}
\newtheorem{proposition}[proposition]{Proposition}
\aliascntresetthe{proposition}

\newaliascnt{lemma}{theorem}
\newtheorem{lemma}[lemma]{Lemma}
\aliascntresetthe{lemma}

\newaliascnt{corollary}{theorem}
\newtheorem{corollary}[corollary]{Corollary}
\aliascntresetthe{corollary}

\newaliascnt{claim}{theorem}

\aliascntresetthe{claim}

\theoremstyle{definition}

\newaliascnt{definition}{theorem}
\newtheorem{definition}[definition]{Definition}
\aliascntresetthe{definition}

\theoremstyle{remark}

\newaliascnt{remark}{theorem}
\newtheorem{remark}[remark]{Remark}
\aliascntresetthe{remark}

\newaliascnt{example}{theorem}
\newtheorem{example}[example]{Example}
\aliascntresetthe{example}

\usepackage{hyperref}
\usepackage[nameinlink]{cleveref}

\crefname{theorem}{Theorem}{Theorems}
\Crefname{theorem}{Theorem}{Theorems}

\crefname{proposition}{Proposition}{Propositions}
\Crefname{proposition}{Proposition}{Propositions}

\crefname{lemma}{Lemma}{Lemmas}
\Crefname{lemma}{Lemma}{Lemmas}

\crefname{corollary}{Corollary}{Corollaries}
\Crefname{corollary}{Corollary}{Corollaries}

\crefname{claim}{Claim}{Claims}
\Crefname{claim}{Claim}{Claims}

\crefname{definition}{Definition}{Definitions}
\Crefname{definition}{Definition}{Definitions}

\crefname{remark}{Remark}{Remarks}
\Crefname{remark}{Remark}{Remarks}

\crefname{example}{Example}{Examples}
\Crefname{example}{Example}{Examples}

\Crefformat{equation}{#2(#1)#3}
\newcommand{\R}{\mathbb{R}}
\newcommand{\N}{\mathbb{N}}
\newcommand{\one}{\mathbf{1}}
\newcommand{\Span}{\operatorname{Span}}
\newcommand{\Par}{\operatorname{Par}}
\newcommand{\Ch}{\operatorname{Ch}}

\newcommand{\cA}{\mathcal{A}}

\newcommand{\Etwo}{E_2}
\newcommand{\Gvis}{G_{\mathrm{vis}}}
\newcommand{\CD}{\mathrm{CD}}
\newcommand{\diam}{\operatorname{diam}}

\newcommand{\card}[1]{\lvert#1\rvert}
\newcommand{\vis}{\mathrm{vis}}
\newcommand{\col}{\mathrm{col}}
\newcommand{\TensorA}[4]{ \mathcal{A}(#1,#2,#3,#4)}
\title{Structure theorems for Lichnerowicz-sharp graphs}
\author{%
Yanlong Ding\thanks{School of Mathematical Sciences, University of Science and Technology of China, Hefei 230026, China. Email address: \texttt{dylustc@mail.ustc.edu.cn}.}%
\and
Shiping Liu\thanks{School of Mathematical Sciences, University of Science and Technology of China, Hefei 230026, China. Email address: \texttt{spliu@ustc.edu.cn}.}%
\and
Chiyu Zhou\thanks{School of Mathematical Sciences, University of Science and Technology of China, Hefei 230026, China. Email address: \texttt{dovong@mail.ustc.edu.cn}.}%
}

\date{}
\begin{document}
\maketitle

\begin{abstract}
Hypercube graphs are fundamental model spaces of positive curvature in discrete comparison geometry.  Let $G$ be a finite, connected, simple, unweighted graph with Bakry--\'Emery curvature bounded below by \(K\).  We call $G$ Lichnerowicz-sharp if its first non-zero non-normalized Laplacian eigenvalue $\lambda_1=K$. We prove that,  after removing a canonical collection of edges on which every \(K\)-eigenfunction is constant, the resulting graph has a canonical bundle structure. Its fibers are regular, have similar structure with hypercubes, and are Laplacian-cospectral with hypercubes, although they need not themselves be hypercubes. If the base graph is nontrivial, then it satisfies \(\CD(K,\infty)\) and has first eigenvalue strictly greater than \(K\).

As a consequence, if the vertex degree in \(G\) is constant along each canonical fiber, then every fiber is a hypercube and \(G\) is a hypercube bundle. Conversely, for every $d\geq 4$, we construct Lichnerowicz-sharp graphs with non-hypercube canonical fibers of degree $d$.

\end{abstract}


\section{Introduction}

Hypercube graphs are fundamental discrete structures with important roles in
geometry \cite{Gromov1999}, probability \cite{Bobkov1997}, and graph theory
\cite{Harary1988}.  From the viewpoint of discrete comparison geometry, they
are basic positively curved model spaces.  In this paper we study the equality
case in the  Lichnerowicz estimate via Bakry--\'Emery curvature and ask what structures are forced by sharpness.

\subsection{Background and motivation}

The continuous theory provides the guiding analogy.  Let \((M^n,g)\) be a
closed Riemannian manifold with
\(\operatorname{Ric}_g\geq (n-1)\kappa g\), where \(\kappa>0\).  The classical
Lichnerowicz theorem gives \(\lambda_1(M)\geq n\kappa\), while Obata's theorem
shows that equality characterizes the round sphere of curvature \(\kappa\)
\cite{Lichnerowicz1958,Obata1962}.  This result is a prototype of spectral
rigidity: a lower curvature bound first yields an analytic inequality, and an
eigenfunction attaining equality satisfies additional identities that recover
global geometry.  Eigenvalue comparison, rigidity, and quantitative pinching
results developed this principle further
\cite{Cheng1975,Petersen1999,Aubry2005}.

In contrast to the Riemannian situation above, sharp equality for the first
eigenvalue alone does not imply rigidity in K\"ahler geometry.  Let
\((M,\omega)\) be a compact K\"ahler manifold of complex dimension \(n\geq2\)
with \(\operatorname{Ric}(\omega)\geq\omega\), and write
\(0<\lambda_1\leq\lambda_2\leq\cdots\) for the eigenvalues of the complex
Laplacian.  Although \(\lambda_1\geq1\), the K\"ahler--Einstein product
$
 (\mathbb{CP}^{n-1},\omega_{\mathbb{CP}^{n-1}})
 \times(\mathbb{CP}^{1},\omega_{\mathbb{CP}^{1}})
$
satisfies \(\operatorname{Ric}(\omega)=\omega\) and
\(\lambda_1=\cdots=\lambda_{n^2+2}=1\), but is not biholomorphically
isometric to \(\mathbb{CP}^{n}\).  Chu, Wang, and Zhang proved that the missing
rigidity is recovered by sufficiently large multiplicity: if
\(\lambda_{n^2+3}=1\), then \((M,\omega)\) is biholomorphically isometric to
\((\mathbb{CP}^{n},\omega_{\mathbb{CP}^{n}})\).  The preceding product shows
that $n^2+3$ is optimal \cite[Theorem~1.6 and Section~6]{CWZ2025}.

The Bakry--\'Emery \(\Gamma\)-calculus and the associated
curvature-dimension conditions \cite{BE1985} provide a synthetic framework
for Ricci curvature; see \cite{BGL2014} for a systematic account.
This framework admits a natural formulation for graph Laplacians and has
become an effective tool for studying geometric, analytic, and combinatorial
properties of graphs; see, for instance,
\cite{BHLLMY2015,FathiShu2018,HornJGT,Hua2019,HM2024,MM2024,
KMY2021,LMP2018,MuenchRose2020,SalezGAFA,SalezJEMS}
and the references therein.

Let $G=(V,E)$ be a finite connected simple unweighted graph.  We use the
non-normalized Laplacian
\begin{equation*}
 Lh(x)=\sum_{y\sim x}\bigl(h(y)-h(x)\bigr).
\end{equation*}
The eigenvalues of $-L$ are ordered as
\[
  0=\lambda_0<\lambda_1\leq \lambda_2\leq\cdots,
\]
where the eigenvalues are counted with multiplicity. As in the continuous
setting, if a finite graph satisfies the Bakry--\'Emery
curvature-dimension condition \(\CD(K,\infty)\) for some \(K>0\), then
the first nonzero eigenvalue of \(-L\) satisfies the discrete
Lichnerowicz estimate
\begin{equation}
\label{eq:Lich}
 \lambda_1\geq K;
\end{equation}
see \cite{curvatureaspectsofgraphs,curvatureandhigherorder}.  For the definition of the Bakry--\'Emery curvature dimension condition $\CD(K,\infty)$, we refer to \Cref{def:CD} below.

The hypercube is the basic model for this equality theory.  The
\(D\)-dimensional hypercube \(H_D\) is the Cartesian product of $D$ copies of $K_2$. It satisfies \(\CD(2,\infty)\), see \cite[Example 7.15]{CLP2020},   and its Laplacian spectrum begins with
\[
 2=\lambda_1(H_D)=\cdots=\lambda_D(H_D)<\lambda_{D+1}(H_D).
\]

Liu, M\"unch, and Peyerimhoff proved that the hypercube is rigid when the
curvature lower bound is attained with sufficiently large multiplicity:
\begin{theorem}[{\cite[Theorem~1.4]{LMP}}]\label{thm:LMP-rigidity}
Suppose that $G$ satisfies $\CD(K,\infty)$ with $K>0$ and $\lambda_{\Delta}=K$, where $\Delta=\max_x\deg(x)$. Then $G\cong H_\Delta$ and, in particular, $K=2$.
\end{theorem}

This rigidity theorem naturally raises the question of what structure
remains when the curvature bound is attained with smaller multiplicity.
In the regular setting, Ding, Liu, and Zhou obtained the following result.
\begin{theorem}[{\cite[Theorem~1.3]{ding2026}}]
  \label{thm:intro-BundleStructure}
  Let $G$ be a connected \(d\)-regular graph satisfying $\CD(K,\infty)$ with $K>0$. Assume $\lambda_1(G)=K$.   Let $m_K(G)$ denote the multiplicity of the $K$-eigenvalue. Then there exists a graph \(G^\prime\) and an integer \(r\geq m_K(G)\) such that \(G\) is an \(H_r\)-bundle over \(G^\prime\). The base graph \(G^\prime\) satisfies \(\CD(K,\infty)\), and if it is not a singleton graph, then \(\lambda_1(G^\prime)>K\). 
\end{theorem}

We next recall the notion of a graph bundle appearing in this theorem.
Graph bundles generalize both graph coverings and Cartesian products and
have been studied extensively from structural, topological, and algorithmic
perspectives; see, for example,  \cite{PTSJVJ1983,MPTSM1988,KJL1990,SohnLee1994,ChaeKwakLee1993,ImrichPisanskiZerovnik1997,HongKwakLee1999,KlavzarMohar1995,KwakLeeSohn1996,ZmazekZerovnik2000,ZmazekZerovnik2002,ZmazekZerovnik2002a,ZmazekZerovnik2006,Zerovnik2000,BanicErvesZerovnik2009,BanicZerovnik2010,ErvesZerovnik2013,FengKwak2006,KwakKwon2001,PisanskiZerovnik2009}; see also \cite{LL2024} for recent developments.
\begin{definition}[{\cite[Definiation 1]{LL2024}}]
    Let $G$ and $F$ be two unweighted graphs. Let $E^{\mathrm{ori}}_{G}=\{(x,y): x\sim y \text{ in } G\}$ be the set of oriented edges of $G$. Let $\sigma: E^{\mathrm{ori}}_G\to \operatorname{Aut}(F), (x,y)\to \sigma_{xy}$ such that $\sigma_{yx}=\sigma_{xy}^{-1}$. Construct the graph $G\square_\sigma F$ by setting $V(G\square_\sigma F)=V(G)\times V(F)$ and $E(G\square_\sigma F)$
    \begin{enumerate}[label=\textup{(\roman*)}]
        \item $(x,u)\sim (y,\sigma_{xy}(u))$ if $x\sim y$ in $G$;
        \item $(x,u)\sim (x,v)$ if $u\sim v$ in $F$.
    \end{enumerate}
    We call $G\square_\sigma F$ the $F$-bundle over $G$ and refer to $G$ and $F$ as the base graph and the fiber graph, respectively.
\end{definition}
In this paper, we extend Ding, Liu, and Zhou's \Cref{thm:intro-BundleStructure} beyond the
regular setting by establishing a bundle structure theorem for arbitrary
Lichnerowicz-sharp graphs.

\subsection{Main results}
Unless otherwise stated, all graphs considered in this paper are finite,
connected, simple, and unweighted.  We call \(G\)
\emph{Lichnerowicz-sharp} if it satisfies \(\CD(K,\infty)\) and
\(\lambda_1(G)=K\) for some \(K>0\).  Unlike in the manifold setting, where
the value of \(K\) depends on the normalization of the metric, the equality case in the unweighted graph setting
forces \(K\) to take a universal value.   More
precisely, Ding, Liu, and Zhou proved the following.
\begin{theorem}[{\cite[Theorem~4.1]{ding2026}}]
\label{thm:first-sharp-K2}
Let \(G\) be a graph satisfying \(\CD(K,\infty)\) with \(K>0\).  If $\lambda_1(G)=K$, then \(K=\mathcal{K}_\infty(x)=2\) for every vertex \(x\in V(G)\).
\end{theorem}
For the definition of $\mathcal{K}_\infty(x)$, see \Cref{def:CD}. Accordingly, the central object is the first eigenspace
\[
 E_2(G)=\ker(L_G+2I).
\]

A basis of \(E_2(G)\) defines a spectral embedding
\(\Phi:V(G)\to\mathbb R^{m_2(G)}\).  This embedding canonically partitions the edge set into two classes: an edge
\(xy\) is called \emph{visible} if \(\Phi(x)\neq\Phi(y)\), and
\emph{collapsed} if \(\Phi(x)=\Phi(y)\). This classification
is independent of the chosen basis.  Let \(G_{\mathrm{vis}}\) denote the
spanning subgraph of \(G\) consisting of the visible edges.

The first structural result describes every connected component of
\(G_{\mathrm{vis}}\) as a rigidly  layered graph.

\begin{theorem}
\label{thm:intro-visible-layer}
Let \(G\) be a Lichnerowicz-sharp graph and let \(H\) be a connected component
of \(G_{\mathrm{vis}}\).  Then there exists a vertex \(p\in V(H)\) with the
following properties.  Put
\[
 D=\deg_H(p),\qquad L_k=S_k^H(p)\quad(0\leq k\leq D),
\]
\begin{enumerate}
 \item \label{en:VerticesNumber} The sets \(L_0,\ldots,L_D\) form a partition of \(V(H)\), and
 \[
  |L_k|=\binom{D}{k}\quad(0\leq k\leq D),
  \qquad |V(H)|=2^D.
 \]

 \item Every \(x\in L_k\) has exactly \(k\) neighbors in \(L_{k-1}\), exactly
 \(D-k\) neighbors in \(L_{k+1}\), and no neighbors in $L_k$. 

 \item\label{en:ChildrenParent} For $1\leq k\leq D-1$ and any two distinct vertices \(x,y\in L_k\),
 \[
  \bigl|S_1^H(x)\cap S_1^H(y)\cap L_{k-1}\bigr|
  =\bigl|S_1^H(x)\cap S_1^H(y)\cap L_{k+1}\bigr|
  \in\{0,1\}.
 \]
\end{enumerate}
\end{theorem}

For a regular Lichnerowicz-sharp graph, every connected component of
\(G_{\mathrm{vis}}\) is a hypercube.  In the irregular setting, this need not
be the case: \(G_{\mathrm{vis}}\) may have connected components that are not
isomorphic to hypercubes.  A concrete example is the graph \(H\) in \Cref{ex:four-dimensional-star-twist}; see
\Cref{sec:NonHypercubeFibers} for a general family of such examples.
\begin{figure}[htbp]
  \centering

\begin{tikzpicture}[
  x=0.75cm,
  y=0.62cm,
  graph edge/.style={draw=black,line width=0.55pt},
  graph vertex/.style={circle,draw=black,fill=white,line width=0.55pt,
    minimum size=4.8pt,inner sep=0pt}
]
\begin{scope}
  \coordinate (hc-p) at (0,0);
  \coordinate (hc-q) at (6,0);
  \foreach \i/\y in {1/1.5,2/0.5,3/-0.5,4/-1.5}{
    \coordinate (hc-a\i) at (1.35,\y);
  }
  \foreach \i/\y in {4/1.5,3/0.5,2/-0.5,1/-1.5}{
    \coordinate (hc-b\i) at (4.65,\y);
  }
  \foreach \e/\y in {12/2,13/1.2,14/0.4,23/-0.4,24/-1.2,34/-2}{
    \coordinate (hc-e\e) at (3,\y);
  }
  \foreach \i in {1,2,3,4}{
    \draw[graph edge] (hc-p)--(hc-a\i);
    \draw[graph edge] (hc-b\i)--(hc-q);
  }
  \foreach \e/\i/\j in {12/1/2,13/1/3,14/1/4,23/2/3,24/2/4,34/3/4}{
    \draw[graph edge] (hc-a\i)--(hc-e\e);
    \draw[graph edge] (hc-a\j)--(hc-e\e);
  }
  \foreach \e/\i/\j in {12/3/4,13/2/4,14/2/3,23/1/4,24/1/3,34/1/2}{
    \draw[graph edge] (hc-e\e)--(hc-b\i);
    \draw[graph edge] (hc-e\e)--(hc-b\j);
  }
  \foreach \v in {hc-p,hc-a1,hc-a2,hc-a3,hc-a4,
    hc-e12,hc-e13,hc-e14,hc-e23,hc-e24,hc-e34,
    hc-b1,hc-b2,hc-b3,hc-b4,hc-q}{
    \node[graph vertex] at (\v) {};
  }
  \node at (3,-2.7) {$4$-dimensional hypercube $H_4$};
\end{scope}

\begin{scope}[xshift=7.8cm]
  \coordinate (st-p) at (0,0);
  \coordinate (st-q) at (6,0);
  \foreach \i/\y in {1/1.5,2/0.5,3/-0.5,4/-1.5}{
    \coordinate (st-a\i) at (1.35,\y);
    \coordinate (st-b\i) at (4.65,\y);
  }
  \foreach \e/\y in {12/2,13/1.2,14/0.4,23/-0.4,24/-1.2,34/-2}{
    \coordinate (st-e\e) at (3,\y);
  }
  \foreach \i in {1,2,3,4}{
    \draw[graph edge] (st-p)--(st-a\i);
    \draw[graph edge] (st-b\i)--(st-q);
  }
  \foreach \e/\i/\j in {12/1/2,13/1/3,14/1/4,23/2/3,24/2/4,34/3/4}{
    \draw[graph edge] (st-a\i)--(st-e\e);
    \draw[graph edge] (st-a\j)--(st-e\e);
    \draw[graph edge] (st-e\e)--(st-b\i);
    \draw[graph edge] (st-e\e)--(st-b\j);
  }
  \foreach \v in {st-p,st-a1,st-a2,st-a3,st-a4,
    st-e12,st-e13,st-e14,st-e23,st-e24,st-e34,
    st-b1,st-b2,st-b3,st-b4,st-q}{
    \node[graph vertex] at (\v) {};
  }
  \node at (3,-2.7) {$H$};
\end{scope}
\end{tikzpicture}
  \caption{Comparison $H$ with $H_4$}
\end{figure}
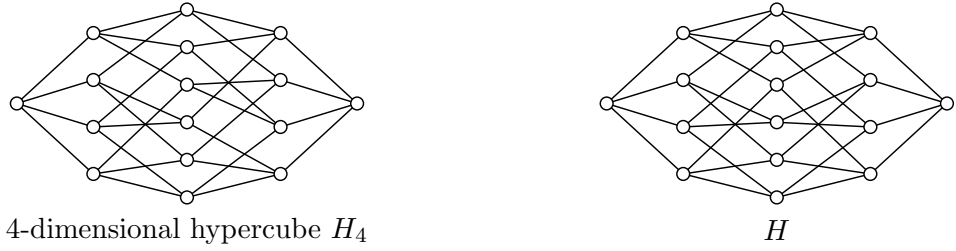

Connected components of \(G_{\mathrm{vis}}\) need not be isomorphic to
hypercubes. Nevertheless, the preceding structural properties imply that
they are adjacency- and Laplacian-cospectral with the corresponding
hypercubes. In this paper, we call a connected graph \(\widetilde G\) a
twisted \(D\)-dimensional hypercube if there exists a vertex \(p\in V(\widetilde G)\) such
that, with
\[
D=\deg_{\widetilde G}(p),
\qquad
L_k=S_k^{\widetilde G}(p)\quad (0\leq k\leq D),
\]
conditions \ref{en:VerticesNumber}--\ref{en:ChildrenParent} in \Cref{thm:intro-visible-layer} hold. We then
obtain the following cospectrality theorem.
\begin{theorem}
  Let $\widetilde{G}$ be a twisted $D$-dimensional hypercube. Then $\widetilde{G}$ is adjacency- and Laplacian-cospectral with $D$-dimensional hypercube.
\end{theorem}
This cospectrality result is also of independent interest in spectral graph theory.

Let \(G'\) be the quotient whose
vertices are the connected components of \(G_{\mathrm{vis}}\), with two
components adjacent when an edge of \(G\) joins them.  Let \(\hat G\) be
obtained from \(G\) by deleting precisely the collapsed edges whose endpoints
lie in the same component of \(G_{\mathrm{vis}}\). We state our main result of this paper.
\begin{theorem}
        \label{thm:BundleStructure}
        Let $G$ be a Lichnerowicz-sharp graph. Let $H$ be a connected component of $G_{\vis}$. Then $\hat{G}$ is a $H$-bundle over $G^\prime$ and admits a section. Moreover, if $G^\prime$ is not a singleton graph, then $G^\prime$ is also $\CD(2,\infty)$ with $\lambda_1(G^\prime)>2$. 
\end{theorem}
Under a natural fiberwise degree condition, the preceding bundle structure
strengthens to a hypercube-bundle structure on \(G\).
\begin{theorem}[Hypercube bundle]
        \label{thm:hypercube-bundle}
        Let $G$ be a Lichnerowicz-sharp graph. If $\deg_G$ is constant on every
        connected component of $G_{\vis}$, then $G$ is a
        hypercube bundle.
\end{theorem}
 Thus, \Cref{thm:BundleStructure} extends 
\Cref{thm:intro-BundleStructure} in regular case to the irregular setting. In general,
however, one cannot replace \(\hat G\) by \(G\) in \Cref{thm:BundleStructure}; see
\Cref{ex:bundle-exception} for a counterexample.

\subsection{Proof strategy}

The proof in the regular setting \cite{ding2026} relies on \emph{exact square
incidence}.  Ding, Liu, and Zhou prove that if \(xy\) is visible and
\(z\in S_1^G(x)\setminus\{y\}\), then \(y\not\sim z\) and there exists a
unique vertex \(w\in S_2^G(x)\) adjacent to both \(y\) and \(z\), with
\[
 S_1^G(x)\cap S_1^G(w)=\{y,z\};
\]
see \cite[Definition~5.7 and Proposition~5.14]{ding2026}.  These squares give
canonical transports of neighbourhoods across edges.  Iterating these transports shows that the connected components of
\(G_{\mathrm{vis}}\) are hypercubes and produces the transition maps
of a hypercube bundle.  

This mechanism does not survive without regularity: the exact-square
conclusion itself is false, rather than merely unavailable to the proof.  For
example, let \(G=K_4\setminus\{e\}\), where \(e=ab\), and denote the other
two vertices by \(c\) and \(d\).  This graph satisfies \(\CD(2,\infty)\) and
\(\lambda_1(G)=2\), while \(E_2(G)\) is spanned by the function taking the
values \(1,-1,0,0\) at \(a,b,c,d\), respectively.  Hence \(ac\) and \(ad\)
are visible, whereas \(cd\) is collapsed.  Thus the visible edge \(ac\) lies
in the triangle with vertices \(a,c,d\), contradicting the nonadjacency
required by exact square incidence.  Consequently, neither the
hypercube structure of the fibers nor the transition maps can be obtained by
propagating individual squares.

Our replacement begins with identities that remain valid without regularity.
Sharpness implies that \(\Gamma(a,b)\) is constant for all
\(a,b\in E_2(G)\) and gives the midpoint and collapsed-edge averaging formulas
\eqref{eq:midpoint-scalar} and \eqref{eq:collapsed-average}.  Choose a $2$-eigenfunction $f$
, which takes distinct values at the endpoints of every
visible edge. Through these identities, we show that \(f\) has a unique maximum and a unique minimum
on every connected component of \(G_{\mathrm{vis}}\)
(\Cref{thm:unique-maximum}). The maximum in each component therefore serves as a root for the subsequent
structural analysis.

Fix such a maximum \(p\) in a connected component \(H\) of
\(G_{\mathrm{vis}}\), put \(D=\deg_H(p)\), and let \(L_k=S_k^H(p)\).  Starting
from \(L_0=\{p\}\), we construct the distance layers inductively.  Suppose
that the required structure is known through \(L_k\).  The second- and
third-moment identities \eqref{eq:w-second-moment} and
\eqref{eq:w-moment-sum} first bound the number of neighbours that a vertex in
\(L_{k+1}\) can have in \(L_k\).  Then formulas
\eqref{eq:edge-layer-count} and \eqref{eq:Gamma-edge-count} force this bound to
be equal. Substituting the resulting equality back into the moment identities
gives \eqref{eq:parent-child-count}.  Thus every \(x\in L_k\) has exactly \(k\)
neighbours in \(L_{k-1}\) and \(D-k\) neighbours in \(L_{k+1}\).

For \(x\in L_k\), write \(\Par(x)\) and \(\Ch(x)\) for its neighbours in
\(L_{k-1}\) and \(L_{k+1}\), respectively.  The identity
\eqref{eq:translated-balance}, together with the inductive uniqueness of a
common child, gives \eqref{eq:edge-iff-child}.  Hence, for distinct
\(x,y\in L_k\),
\[
 |\Par(x)\cap\Par(y)|
 =|\Ch(x)\cap\Ch(y)|\in\{0,1\}.
\]
This completes the induction and proves the layered structure in
\Cref{thm:intrinsic-Gamma}.

It remains to construct the bundle without square-based transport.  For two
adjacent components \(C\) and \(D\), let \(M\) record the edges between them.
The midpoint identity gives the relation
\[
 A_CM=MA_D.
\]
Perron--Frobenius theory, the structure of the components, and uniqueness of
their extrema force \(M\) to be a permutation matrix and its associated
bijection to be an isomorphism (\Cref{thm:component-matching}).  These
isomorphisms provide the transition maps of the bundle \(\hat G\to G'\), and
the componentwise maxima form a section. 

\subsection{The use of AI}
In part of this work, the authors were assisted by AI tools. In particular, we were assisted to identify the second- and third-order tensors in \Cref{eq:SecondOrder} and \Cref{eq:ThirdOrder}, which play an important role in the proof of our main result. We are also assisted to constructing many important examples that have significantly deepened our understanding of the problem.

All mathematical arguments, computations, and conclusions were independently verified by the authors. No text in this article was written by AI.

\section{Preliminaries}

Let $G=(V,E)$ be a finite connected simple unweighted graph.  We use the
non-normalized Laplacian
\begin{equation*}
 Lh(x)=\sum_{y\sim x}\bigl(h(y)-h(x)\bigr).
\end{equation*}
The eigenvalues of $-L$ are ordered as
\[
  0=\lambda_0<\lambda_1\leq \lambda_2\leq\cdots,
\]
with multiplicity.  For $K>0$, put
\[
  E_K:=\ker(L+KI),\quad m_K:=\dim E_K.
\]
We write $d_G$ for the combinatorial distance, and
\[
        S_r(x):=\{y\in V:d_G(x,y)=r\},\quad
        B_r(x):=\{y\in V:d_G(x,y)\leq r\}.
\]
For any two adjacent vertices $x,y$, we write $xy\in E$ or $x\sim y$.

The $d$-dimensional hypercube $H_d$ is the graph with vertex set $ \mathbb{F}_2^d$, where two vertices are adjacent if and only if they differ in exactly one coordinate; $H_0$ denotes the one-vertex graph.

For functions $f,g,h:V\to\R$, define
\begin{align}
 \label{eq:Gamma-def}
 \Gamma(f,g)(x)
 &=\frac12\sum_{y\sim x}
   \bigl(f(y)-f(x)\bigr)\bigl(g(y)-g(x)\bigr),\\
  \label{eq:Gamma2-def}
 \Gamma_2(f,g)
 &=\frac12\bigl(L\Gamma(f,g)-\Gamma(g,Lf)-\Gamma(f,Lg)\bigr),\\
 T_3(f,g,h)(x)&=\frac{1}{2}\sum_{y\sim x}(f(y)-f(x))(g(y)-g(x))(h(y)-h(x)).
\end{align}
Write $\Gamma(g)=\Gamma(g,g)$ and $\Gamma_2(g)=\Gamma_2(g,g)$.

 We have the following product formula for $\Gamma$.
 \begin{lemma}
     Let $f,g,h: V\to\mathbb{R}$ be three functions. Then
     \begin{equation}
     \label{eq:ProductFormulaOfGamma}
         \Gamma(f,gh)=g\Gamma(f,h)+h\Gamma(f,g)+T_3(f,g,h).
     \end{equation}
 \end{lemma}
 \begin{proof}
     \begin{equation*}
     \begin{aligned}
         \Gamma(f,gh)(x)&=\frac{1}{2}\sum_{y\sim x}(f(y)-f(x))(g(y)h(y)-g(x)h(x))\\
         &=g(x)\Gamma(f,h)(x)+\frac{1}{2}\sum_{y\sim x}h(y)(f(y)-f(x))(g(y)-g(x))\\
         &=g(x)\Gamma(f,h)(x)+h(x)\Gamma(f,g)(x)+T_3(f,g,h)(x).
     \end{aligned}
     \end{equation*}
 \end{proof}
\begin{definition}
\label{def:CD}
    Let $G=(V,E)$ be a graph. We say that $G$ satisfies the Bakry--\'Emery curvature-dimension condition $\CD(K,\infty)$ at a vertex $x\in V$ if
    \[
  \Gamma_2(f)(x)\ge K\Gamma(f)(x)
\]
holds for any function $f:V\to\R$. The $\infty$-Bakry--\'Emery curvature $\mathcal{K}_{\infty}(x)$ at $x\in V$ is defined as
\[\mathcal{K}_{\infty}(x)=\sup\{K: G \text{ satisfies }\CD(K,\infty)\ \text{at}\ x\}.\]
We say $G$ satisfies $\CD(K,\infty)$ if it satisfies $\CD(K,\infty)$ at every $x\in V$.
\end{definition}

We recall two basic consequences of positive Bakry--\'Emery curvature: the
Lichnerowicz eigenvalue bound and the Myers diameter bound.

\begin{theorem}[Lichnerowicz bound, {\cite[Proposition~1.3]{LMP}}]
\label{thm:LichnerowiczEstimate}
     Assume $G$ satisfies $\CD(K,\infty)$ with $K>0$. Let $D=\max_{v\in V(G)}\deg(v)<\infty$. Then
     \begin{equation}
         \label{eq:LichnerowiczEstimate}
         \lambda_1\geq K.
     \end{equation}
\end{theorem}

\begin{theorem}[{\cite[Proposition~1.3 and Theorem~1.4]{LMP}}]
\label{thm:MyersBound}
    Assume $G$ satisfies $\CD(K,\infty)$ with $K>0$. Let $D=\max_{v\in V(G)}\deg(v)<\infty$. Let $\diam(G)$ be the diameter of $G$. Then $G$ satisfies the Myers diameter bound
    \begin{equation}
    \label{eq:MyersBound}
        \diam(G)\leq\frac{2D}{K}.
    \end{equation}
    The equality holds if and only if $G$ is $H_D$.
\end{theorem}

A key tool for computing the Bakry--\'Emery curvature at a vertex is the so-called curvature matrix. For each vertex $x$ of a graph $G$, one can associate a symmetric $\deg(x)\times \deg(x)$ matrix $A_\infty(x)$, whose entries are determined by the local structure of the ball $B_2(x)$, such that
\begin{equation}
\mathcal{K}_\infty(x)=\lambda_{\min}\bigl(A_\infty(x)\bigr),
\end{equation}
where $\lambda_{\min}\bigl(A_\infty(x)\bigr)$ denotes the smallest eigenvalue of $A_\infty(x)$. This matrix can be obtained from the matrices representing $\Gamma_2$ and $\Gamma$ by taking a Schur complement, see \cite[Section 2]{CKLP}. Below we recall the definition of the curvature matrix $A_\infty(x)$, in a slightly reformulated form of \cite[Proposition 1.13(i)]{CKLP}.

\begin{definition}[Curvature matrices]\label{def:curMatrix}
    Let $G$ be a graph. For any vertex $x$, we label
\[
   S_1(x)=\{y_1,\ldots,y_{\mathrm{deg}(x)}\}.
\]
For $i\neq j$, set
\[
  \varepsilon_{ij}:=\begin{cases}1,&y_i\sim y_j,\\0,&y_i\not\sim y_j,
  \end{cases}
\]
and
\[
  \omega_{ij}:=\sum_{\substack{z\in S_2(x)\\ z\sim y_i,\ z\sim y_j}}
      \frac1{d_x^-(z)}\ \ \text{with}
  \ \ d_x^-(z):=|S_1(x)\cap S_1(z)|.
\]
Let
\begin{equation}
\label{eq:tandOmega}
 t_i:=\sum_{j\neq i}\varepsilon_{ij}\qquad \text{and} \qquad\Omega_i:=\sum_{j\neq i}\omega_{ij}.
\end{equation}
Then we have
\begin{align}
\label{eq:A-offdiag}
      (A_\infty(x))_{ij}&=1-2\varepsilon_{ij}-2\omega_{ij}
  \ \ \text{for}\ \ i\neq j,
\\
\label{eq:CKLP-diag}
       (A_\infty(x))_{ii}&=3-\frac{\mathrm{deg}(x)+\mathrm{deg}(y_i)}{2}+\frac52 t_i+2\Omega_i\\
\end{align}
\end{definition}
By \eqref{eq:A-offdiag} and \eqref{eq:CKLP-diag}, we derive that
\begin{equation}
    \label{eq:A-row-sum}
      (A_\infty(x)\one)_i=2+\frac12 t_i+\frac{\mathrm{deg}(x)-\mathrm{deg}(y_i)}{2},
\end{equation}
where $\one$ is the $\deg(x)$-dimensional all-$1$ vector.

\begin{example}\label{ex:hypercube_cur_matrix}
    The hypercube graph $H_d$ has $\deg(x)=\deg(y_i)=d$, $\varepsilon_{ij}=0$ and $\omega_{ij}=1/2$ for every
$i\neq j$, hence \[A_\infty(x)=2I_d,\]
at any vertex $x$.
\end{example}

\begin{definition}[{\cite[Definition 1]{LL2024}}]
\label{def:GraphBundle}
    Let $G$ and $F$ be two unweighted graphs. Let $E^{\mathrm{ori}}_{G}=\{(x,y): x\sim y \text{ in } G\}$ be the set of oriented edges of $G$. Let $\sigma: E^{\mathrm{ori}}_G\to \operatorname{Aut}(F), (x,y)\to \sigma_{xy}$ such that $\sigma_{yx}=\sigma_{xy}^{-1}$. Construct the graph $G\square_\sigma F$ by setting $V(G\square_\sigma F)=V(G)\times V(F)$ and $E(G\square_\sigma F)$
    \begin{enumerate}
        \item[(\romannumeral1)] $(x,u)\sim (y,\sigma_{xy}(u))$ if $x\sim y$ in $G$;
        \item[(\romannumeral2)] $(x,u)\sim (x,v)$ if $u\sim v$ in $F$.
    \end{enumerate}
    We call $G\square_\sigma F$ the $F$-bundle over $G$ and refer to $G$ and $F$ as the base graph and the fiber graph, respectively.
\end{definition}
\begin{definition}
    Let $G\square_\sigma F$ be a $F$-bundle over $G$. A section of $G\square_\sigma F$ is a graph homomorphism $s: G\to G\square_\sigma F$ such that $\pi_1\circ s: V(G)\to V(G)$ is the identity map where $\pi_1$ is the projection to the first factor $V(G\square_\sigma F)\to V(G), (x,u)\to x$.
\end{definition}

\section{Curvature formulas}

We call a graph $G$ Lichnerowicz-sharp if it satisfies $\CD(K,\infty)$ and $\lambda_1(G)=K$ for some $K>0$. 
In this section, we give some fundamental equations for Lichnerowicz-sharp graphs.  Lemma~\ref{lem:defect-nullity} and Lemma~\ref{lem:midpoint} can be found in \cite{LMP}. 

Unless otherwise stated, we shall always assume that $G$ is a connected Lichnerowicz-sharp graph.

\begin{lemma}[{\cite[Theorem~2.1 and Section~3.1]{LMP}}]
\label{lem:defect-nullity}
For every $a\in E_K$, every vertex $s$, and every function $\psi:V\to\R$, one has
\begin{equation}\label{eq:defect-bilinear-zero}
 \Gamma_2(a,\psi)(s)-K\Gamma(a,\psi)(s)=0.
\end{equation}
Moreover, $\Gamma(a,b)$ is constant for all $a,b\in E_K$.
\end{lemma}

The following two formulas play a fundamental role in the proofs of our structural theorems. Let $\one_{\{x\}}$ be the indicator function of $\{x\}$, i.e., $\one_{\{x\}}(x)=1$ and $\one_{\{x\}}(y)=0$ for all $y\neq x$. The midpoint formula below is contained in \cite[Theorem~3.4]{LMP}. For consistency with the proof of Lemma~\ref{lem:adjacent-edge}, we include a direct computational proof.
\begin{lemma}
\label{lem:midpoint}
If $d_G(x,y)=2$, then for every $a\in E_K$,
\begin{equation}\label{eq:midpoint-scalar}
 \frac{a(x)+a(y)}2
 =\frac1{|S_1(x)\cap S_1(y)|}
   \sum_{\substack{z\sim x,\\z\sim y}}a(z).
\end{equation}
\end{lemma}

\begin{proof}
Fix $a\in E_K$ and put $\psi=\mathbf{1}_{\{y\}}$.  Since $d_G(x,y)=2$,
$\Gamma(a,\psi)(x)=0$.  A direct calculation gives
\begin{align*}
 L\Gamma(a,\psi)(x)
 &=\frac12\sum_{\substack{z\sim x,\\z\sim y}}\bigl(a(y)-a(z)\bigr),\\
 \Gamma(a,L\psi)(x)
 &=\frac12\sum_{\substack{z\sim x,\\z\sim y}}\bigl(a(z)-a(x)\bigr).
\end{align*}
Since \(La=-Ka\) and \(\Gamma(a,\psi)(x)=0\), we also have
\(\Gamma(\psi,La)(x)=-K\Gamma(\psi,a)(x)=0\). Hence
\eqref{eq:defect-bilinear-zero} reduces to
\[
L\Gamma(a,\psi)(x)-\Gamma(a,L\psi)(x)=0.
\]
Substitution yields
\[
 \sum_{\substack{z\sim x,\\z\sim y}}
 \bigl(a(x)+a(y)-2a(z)\bigr)=0,
\]
which is \eqref{eq:midpoint-scalar}.
\end{proof}

\begin{lemma}
\label{lem:adjacent-edge}
Let $xy\in E(G)$ and put $c_{xy}=|S_1(x)\cap S_1(y)|$. Write
$d_x=\deg_G(x)$ and $d_y=\deg_G(y)$.  Then every $a\in E_K$ satisfies
\begin{equation}\label{eq:adjacent-edge}
 (d_x-d_y+c_{xy})a(x)
 +(-d_x+d_y+c_{xy})a(y)
 =2\sum_{\substack{z\sim x,\\z\sim y}}a(z).
\end{equation}
In particular, if $a(x)=a(y)$ and $S_1(x)\cap S_1(y)\ne\varnothing$, then
\begin{equation}\label{eq:collapsed-average}
 a(x)=a(y)
 =\frac1{|S_1(x)\cap S_1(y)|}\sum_{\substack{z\sim x,\\z\sim y}}a(z).
\end{equation}
\end{lemma}

\begin{proof}
Fix $a\in E_K$ and again put $\psi=\mathbf{1}_{\{y\}}$.  A direct computation using $La=-Ka$ gives
\begin{align}
\label{eq:LGamma}
 L\Gamma(a,\psi)(x)
 &=\frac12\left[
 (K+c_{xy}-d_x)a(y)+d_xa(x)-\sum_{\substack{z\sim x,\\z\sim y}}a(z)
 \right],\\
 \label{eq:GammaL}
 \Gamma(a,L\psi)(x)
 &=\frac12\left[
 -d_ya(y)+(d_y-c_{xy}+K)a(x)+\sum_{\substack{z\sim x,\\z\sim y}}a(z)
 \right],\\
 \label{eq:GammaAPsi}
 \Gamma(a,\psi)(x)&=\frac{1}{2}(a(y)-a(x)).
\end{align}
Since $La=-Ka$, equation \eqref{eq:defect-bilinear-zero} is
\[
 L\Gamma(a,\psi)(x)-\Gamma(a,L\psi)(x)
 -K\Gamma(a,\psi)(x)=0.
\]
Substituting \eqref{eq:LGamma}, \eqref{eq:GammaL} and \eqref{eq:GammaAPsi} and collecting coefficients gives
\eqref{eq:adjacent-edge}.  If $a(x)=a(y)$, division by $c_{xy}>0$ gives \eqref{eq:collapsed-average}.
\end{proof}

By \Cref{thm:first-sharp-K2}, we shall always assume \(K=2\) in the rest of this paper.

\section{Spectral embedding and generic 2-eigenfunction}
\subsection{Spectral embedding}
Let $m=m_2(G)$. Choose an orthonormal basis $\phi_1,\ldots,\phi_m$ of $\Etwo$ in
$\ell^2(V)$, and define the spectral embedding
\begin{equation*}
 \Phi(x)=\bigl(\phi_1(x),\ldots,\phi_m(x)\bigr)\in\mathbb{R}^{m}.
\end{equation*}
An edge $xy$ is \emph{visible} if $\Phi(x)\ne\Phi(y)$ and
\emph{collapsed} otherwise.  Define the following subgraph
\begin{equation*}
    G_{\vis}=\bigl(V(G),\{xy\in E(G): xy\text{ is visible}\}\bigr).
\end{equation*}

For any vertex $s\in V(G)$, define the following second and third-order tensors
\begin{align}
 Q(s)&=\sum_{u\sim s}\label{eq:SecondOrder}
 (\Phi(u)-\Phi(s))^{\otimes2},
\\
 T(s)&=\sum_{u\sim s}(\Phi(u)-\Phi(s))^{\otimes3}.\label{eq:ThirdOrder}
\end{align}

\begin{lemma}
\label{lem:frame-constant}
$Q(s)$ is independent of $s$.  Moreover,
\begin{equation}\label{eq:frame-global}
 Q(s)=\frac4{|V|}I_m.
\end{equation}
\end{lemma}

\begin{proof}
For $u,w\in\mathbb{R}^m$, let $f_u(s)=\langle u,\Phi(s)\rangle$.  By
\Cref{lem:defect-nullity}, $\Gamma(f_u,f_w)$ is constant.  Since
\[
 \langle Q(s)u,w\rangle=2\Gamma(f_u,f_w)(s),
\]
$Q(s)$ is constant.  Summing over $s$ and then over unoriented edges gives
\begin{align*}
 |V|\langle Q(s)u,w\rangle
 &=\sum_{z\in V}\sum_{r\sim z}
  (f_u(r)-f_u(z))(f_w(r)-f_w(z))\\
 &=-2\langle f_u,Lf_w\rangle
 =4\langle u,w\rangle,
\end{align*}
which proves \eqref{eq:frame-global}.
\end{proof}

\begin{lemma}[Cubic $2$-eigenfunction]
\label{lem:cubic-sharpness}
For $a,b,c\in\Etwo$, the function
\begin{equation*}
 T_3(a,b,c)(s)=\frac{1}{2}\sum_{u\sim s}
 (a(u)-a(s))(b(u)-b(s))(c(u)-c(s))
\end{equation*}
belongs to $\Etwo$.
\end{lemma}

\begin{proof}
For $b,c\in\Etwo$, the constancy of $\Gamma(b,c)$ implies that
\begin{equation}\label{eq:psi-bc}
 \psi_{b,c}=bc-\frac12\Gamma(b,c)
 \qquad\text{satisfies}\qquad
 L\psi_{b,c}=-4\psi_{b,c}.
\end{equation}
By \eqref{eq:defect-bilinear-zero}, for $a\in\Etwo$ and any $\psi$,
\begin{equation}\label{eq:commutator}
 L\Gamma(a,\psi)-\Gamma(a,L\psi)=2\Gamma(a,\psi).
\end{equation}
Substituting \eqref{eq:psi-bc} into \eqref{eq:commutator} gives
$L\Gamma(a,\psi_{b,c})=-2\Gamma(a,\psi_{b,c})$.  Applying the product formula \eqref{eq:ProductFormulaOfGamma} to $a,b,c$ gives
\begin{equation*}
 \Gamma(a,\psi_{b,c})=\Gamma(a,bc)
 =b\Gamma(a,c)+c\Gamma(a,b)+T_3(a,b,c).
\end{equation*}
The first two terms on the right are constant multiples of $b$ and $c$, hence
belong to $\Etwo$.  Therefore $T_3(a,b,c)\in\Etwo$.
\end{proof}

Denote $T_{u_1u_2u_3}(s):=T(s)(u_1,u_2,u_3)$. Define the fourth-order tensor $\mathcal{A}$ of $\mathbb{R}^{m}$ by
\begin{equation}
    \label{eq:FourthOrderTensor}
    \mathcal{A}(u_0,u_1,u_2,u_3)=\langle u_0,\sum_{i=1}^{m}\langle T_{u_1u_2u_3},\phi_i \rangle e_i\rangle
\end{equation}
For any $u\in\mathbb{R}^{m}$, let $\mathcal{A}(u)$ denote the third-order tensor
\begin{equation*}
    \mathcal{A}(u)=\mathcal{A}(u,\cdot,\cdot,\cdot).
\end{equation*}

\begin{lemma}
\label{lem:A4-symmetry}
The tensor $\mathcal{A}$ has the following properties:
\begin{enumerate}
    \item\label{en:TensorA1} $\mathcal{A}(\Phi(s))=T(s)$;
    \item \label{en:TensorA2}$\mathcal{A}$ has the following formula: For $u_i\in\mathbb{R}^m,\ 0\leq i\leq 3$,
    \begin{equation}\label{eq:A4-edge}
 \cA(u_0,u_1,u_2,u_3)
 =-\sum_{s\sim z}
  \prod_{j=0}^3\bigl(f_{u_j}(s)-f_{u_j}(z)\bigr).
\end{equation}
In particular, $\mathcal{A}$ is symmetric in all four arguments.
\end{enumerate}
\end{lemma}
\begin{proof}
For \ref{en:TensorA1}, we have
\begin{equation*}
\begin{aligned}
      \mathcal{A}(\Phi(s),u_1,u_2,u_3)&=\bigl\langle \Phi(s),\sum_{i=1}^{m}\langle T_{u_1u_2u_3},\phi_i \rangle e_i\bigr\rangle\\
      &=\sum_{i=1}^{m}\langle T_{u_1u_2u_3},\phi_i \rangle\phi_i(s)\\
      &=T_{u_1u_2u_3}(s).
\end{aligned}
\end{equation*}
The last equality comes from $T_{u_1u_2u_3}\in E_2$.

For \ref{en:TensorA2}, we have
\begin{equation*}
\begin{aligned}
      \mathcal{A}(u_0,u_1,u_2,u_3)&=\bigl\langle u_0,\sum_{i=1}^{m}\langle T_{u_1u_2u_3},\phi_i \rangle e_i\bigr\rangle\\
      &=\sum_{i=1}^{m}u_{0i}\langle T_{u_1u_2u_3},\phi_i \rangle \\
      &=\sum_{i=1}^{m}u_{0i}\sum_{s\in V}\phi_i(s)\sum_{z\sim s}\prod_{k=1}^3\langle u_k,\Phi(z)-\Phi(s)\rangle \\
      &=\sum_{s\in V}f_{u_0}(s)\sum_{z\sim s}\prod_{k=1}^{3}\langle u_k,\Phi(z)-\Phi(s)\rangle\\
      &=-\sum_{s\sim z}
  \prod_{k=0}^3\bigl(f_{u_k}(s)-f_{u_k}(z)\bigr).
\end{aligned}
\end{equation*}
\end{proof}

\subsection{Generic 2-eigenfunction}
\begin{lemma}[generic $2$-eigenfunction]
\label{lem:visible-generic}
There exists $f\in\Etwo$ such that
\begin{equation}\label{eq:visible-strict}
 f(x)\ne f(y)
 \qquad\text{for every visible edge }xy.
\end{equation}
\end{lemma}

\begin{proof}
For each visible edge $xy$, the linear functional
$\ell_{xy}(f)=f(x)-f(y)$ on $\Etwo$ is nonzero.  Its kernel is therefore a
proper hyperplane.  A finite union of proper linear hyperplanes cannot cover
$\Etwo$, so one may choose $f$ outside their union.
\end{proof}

Let $f\in\Etwo$ satisfy \eqref{eq:visible-strict}.  For $t\in\R$, define
\begin{equation*}
 H_{>t}:=\Gvis\bigl[\{x\in V(G):f(x)>t\}\bigr].
\end{equation*}
Thus $H_{>t}$ contains precisely the visible edges whose two endpoints have
$f$-value strictly larger than $t$.

\begin{lemma}
\label{lem:four-cycle-exclusion}
Let $f\in\Etwo$ satisfy \eqref{eq:visible-strict}. Assume there exist two vertices $u,v$ lying in different components of $H_{>t}$, such that $uv$ is a collapsed edge and $f(u)=f(v)>t$. Then there exists no visible path $uabv$ with $f(a)=t$ and $f(b)>t$.
\end{lemma}

\begin{proof}
Assume that such vertices exist.  Put
\begin{equation*}
 s=f(u)=f(v),\qquad r=f(b),\qquad
 \alpha=\min\{s,r\},\qquad \beta=\max\{s,r\}.
\end{equation*}
Consequently, $ t<\alpha<\beta$.

Since $bv$ is an edge of $H_{>t}$, the vertices $b$ and $v$ lie in the same
component of $H_{>t}$.  Hence $u$ and $b$ lie in different components of
$H_{>t}$. Since $s\neq r$, if $u$ and $b$ were adjacent in $G$, then $ub$ would be a visible edge. It is impossible. Hence we have $d_G(u,b)=2$.  By \eqref{eq:midpoint-scalar}, we have
\begin{equation}\label{eq:ub-midpoint}
 \frac{s+r}{2}
 =\frac1{|S_1^G(u)\cap S_1^G(b)|}
   \sum_{z\in S_1^G(u)\cap S_1^G(b)}f(z).
\end{equation}

Let $z\in S_1^G(u)\cap S_1^G(b)$.  If $f(z)>t$, the edges $uz$ and $bz$ cannot both
be visible, since they would connect $u$ to $b$ inside $H_{>t}$.  They cannot
both be collapsed either, since that would give $s=f(z)=r$.  Thus exactly one
of the two edges is collapsed, and
\begin{equation*}
 f(z)\in\{s,r\}=\{\alpha,\beta\}
 \qquad\text{whenever }f(z)>t.
\end{equation*}
If $f(z)\le t$, then $f(z)<\alpha$.  Therefore every common neighbour
satisfies
\begin{equation}\label{eq:common-value-dichotomy}
 f(z)\le\alpha
 \qquad\text{or}\qquad
 f(z)=\beta.
\end{equation}
Moreover, the common neighbour $a$ satisfies $f(a)=t<\alpha$.

We claim that at least two common neighbours of $u$ and $b$ have value
$\beta$.  Put $m=|S_1^G(u)\cap S_1^G(b)|$.  The distinct vertices $a,v$ show that
$m\ge2$.  If at most one common neighbour had value $\beta$, then
\eqref{eq:common-value-dichotomy} would imply
\begin{equation}\label{eq:beta-count-bound}
 \frac1m\sum_{z\in S_1^G(u)\cap S_1^G(b)}f(z)
 <\frac{\beta+(m-1)\alpha}{m}.
\end{equation}
For $m=2$, the right-hand side equals $(\alpha+\beta)/2$ and for $m>2$,
\begin{equation*}
 \frac{\beta+(m-1)\alpha}{m}
 <\frac{\alpha+\beta}{2}.
\end{equation*}
In either case \eqref{eq:beta-count-bound} contradicts
\eqref{eq:ub-midpoint}.  Hence there exist distinct vertices
$c_0^1,c_0^2\in S_1^G(u)\cap S_1^G(b)$ with
\begin{equation*}
 f(c_0^1)=f(c_0^2)=\beta.
\end{equation*}

Let $A_{-1}$ be the vertex in $\{u,b\}$ whose value is $\alpha$, and let
$A_0$ be the other one.  Write
\begin{equation*}
 h_{-1}=f(A_{-1})=\alpha,
 \qquad
 h_0=f(A_0)=\beta.
\end{equation*}
The vertices $A_{-1}$ and $A_0$ lie in different components of $H_{>t}$.
For $j=1,2$, the vertex $c_0^j$ is adjacent to both of them. The edges $A_0c_0^j$ are collapsed, whereas $A_{-1}c_0^j$ are visible edges. 

We now construct recursively vertices
\begin{equation*}
 A_{-1},A_0,A_1,A_2,\ldots
 \quad\text{and}\quad
 c_n^1,c_n^2\quad(n\ge0)
\end{equation*}
with the following properties:
\begin{enumerate}
    \item \label{en:recursive-heights} $h_n:=f(A_n), h_{n-1}<h_n$;
     \item  \label{en:recursive-components} $A_{n-1},A_n$ lie in different components of $H_{>t}$;
    \item there exist distinct vertices $c_n^1,c_n^2\in S_1^G(A_{n-1})\cap S_1^G(A_n)$ such that
    \begin{equation}
      \label{eq:c1c2collapsed}
        f(c_n^1)=f(c_n^2)=h_n;
    \end{equation}
    \item  \label{en:recursive-edge-types} $A_nc_n^j$ is collapsed and $A_{n-1}c_n^j$ is visible for $j=1,2$.
\end{enumerate}
The case $n=0$ was established above.

Assume that $A_k$ and $c_k^1,c_k^2$ have been constructed for $0\leq k\leq n$.  The
vertices $c_n^1,c_n^2$ have the same $f$-value and have the common neighbours
$A_{n-1},A_n$.  If $c_n^1\sim_G c_n^2$ is an edge, it is collapsed by
\eqref{eq:c1c2collapsed}, and \Cref{lem:adjacent-edge} applies; if it is
not an edge, the two vertices have distance two in $G$ and
\Cref{lem:midpoint} applies.  In both cases the average of $f$ over their
common neighbours equals $h_n$:
\begin{equation*}
 h_n=
 \frac1{|S_1^G(c_n^1)\cap S_1^G(c_n^2)|}
 \sum_{z\in S_1^G(c_n^1)\cap S_1^G(c_n^2)}f(z).
\end{equation*}
Since $A_{n-1}$ is a common neighbour with value $h_{n-1}<h_n$, there is a
common neighbour $A_{n+1}$ satisfying
\begin{equation*}
 h_{n+1}:=f(A_{n+1})>h_n.
\end{equation*}
Both edges $A_{n+1}c_n^j$ are visible.  Thus $A_{n+1}$ is connected in
$H_{>t}$ to $A_{n-1}$ through $c_n^j$.  By
\ref{en:recursive-components}, $A_n$ and $A_{n+1}$ lie in different
components of $H_{>t}$.

The vertices $A_n,A_{n+1}$ cannot be adjacent in $G$: a collapsed edge would
contradict $h_n<h_{n+1}$, while a visible edge would force them to lie in the same component of $H_{>t}$.  They have the two common neighbours $c_n^1,c_n^2$,
so \Cref{lem:midpoint} gives
\begin{equation}\label{eq:An-midpoint}
 \frac{h_n+h_{n+1}}2
 =\frac1{|S_1^G(A_n)\cap S_1^G(A_{n+1})|}
   \sum_{z\in S_1^G(A_n)\cap S_1^G(A_{n+1})}f(z).
\end{equation}
As before, every common neighbour $z$ with $f(z)>t$ must be joined to exactly
one of $A_n,A_{n+1}$ by a collapsed edge.  Hence
\begin{equation*}
 f(z)\le h_n
 \qquad\text{or}\qquad
 f(z)=h_{n+1}.
\end{equation*}
The two common neighbours $c_n^1,c_n^2$ both have value $h_n$.  If at most one
common neighbour had value $h_{n+1}$, then the average in
\eqref{eq:An-midpoint} would be strictly smaller than
$(h_n+h_{n+1})/2$: for two common neighbours it is at most $h_n$, and for
$m\ge3$ common neighbours it is at most
\[
 \frac{h_{n+1}+(m-1)h_n}{m}
 <\frac{h_n+h_{n+1}}2.
\]
This contradiction produces two distinct common neighbours
$c_{n+1}^1,c_{n+1}^2$ of value $h_{n+1}$.  The assertions in
\ref{en:recursive-edge-types} hold.  The construction is complete.

\ref{en:recursive-heights} yields the strictly increasing sequence
\[
 h_{-1}<h_0<h_1<h_2<\cdots.
\]
Consequently the vertices $A_{-1},A_0,A_1,\ldots$ are pairwise distinct,
contradicting the finiteness of $G$.  This proves the lemma.
\end{proof}

\begin{lemma}
\label{lem:collapsed-superlevel-detour}
Let $uv$ be a collapsed edge, let $t\in\R$, and suppose
\begin{equation*}
 f(u)=f(v)>t.
\end{equation*}
If there exists a visible  path from $u$ to $v$, all of whose vertices have value at
least $t$, then there exists a visible path from $u$ to $v$ all of whose vertices
have value strictly larger than $t$.
\end{lemma}

\begin{proof}
Assume $u$ and $v$ do not lie in the same connected component of $H_{>t}$. Among all
visible walks from $u$ to $v$ whose vertices have value at least $t$, choose
one,
\begin{equation*}
 P=(u=v_0,v_1,\ldots,v_m=v),
\end{equation*}
minimizing $\card{\{i: f(v_i)=t\}}$.  This
number is positive.  Choose an index $j$ with $f(v_j)=t$, and put
\begin{equation*}
 a=v_{j-1},\qquad b=v_j,\qquad c=v_{j+1}.
\end{equation*}
The endpoints of $P$ have value greater than $t$, and adjacent vertices of a
visible walk have different $f$-values.  Hence
\begin{equation*}
 f(a)>t,
 \qquad
 f(c)>t.
\end{equation*}

First, we have $a,c$ lie in different components of $H_{>t}$. Otherwise replacing $abc$ by the path joining $a$ and $c$ in $H_{>t}$ contradicts the
minimality of $P$. Thus if $a$ is adjacent to $c$ in $G$, then $ac$ must be collapsed. By \Cref{lem:adjacent-edge},
\begin{equation*}
 f(a)=f(c)=\frac1{|S_1^G(a)\cap S_1^G(c)|}
   \sum_{z\in S_1^G(a)\cap S_1^G(c)}f(z).
\end{equation*}
The common neighbour $b$ has value $t<f(a)$, so another common neighbour $z$
has value $f(z)>f(a)$.  Then $azc$ is a path in $H_{>t}$, a contradiction.

It remains that $a$ and $c$ are nonadjacent in $G$.  By
\Cref{lem:midpoint},
\begin{equation}\label{eq:ac-midpoint-detour}
 \frac{f(a)+f(c)}2
 =\frac1{|S_1^G(a)\cap S_1^G(c)|}
   \sum_{z\in S_1^G(a)\cap S_1^G(c)}f(z).
\end{equation}
Since $b$ is a common neighbour with value below the left-hand side, there is
a common neighbour $w\ne b$ such that
\begin{equation*}
 f(w)>\frac{f(a)+f(c)}2>t.
\end{equation*}
Then exactly one of $aw,cw$ is collapsed and the other one is visible. 

After interchanging $a$ and $c$,
assume that $aw$ is collapsed and $cw$ is visible.  Then $abcw$ is a visible path, its vertices have value at least $t$, its endpoints are
joined by the collapsed edge $aw$, and
\begin{equation}
\label{eq:SquareWithContradiction}
 f(a)=f(w)>t,
 \qquad
 f(b)=t,
 \qquad
 f(c)>t.
\end{equation}
Since $a$ and $c$ lie in different components of $H_{>t}$ and $wc$ is visible, $a,w$ don't lie in the same component of $H_{>t}$. \eqref{eq:SquareWithContradiction} contradicts \Cref{lem:four-cycle-exclusion}.  The contradiction proves the lemma.
\end{proof}

\begin{lemma}
\label{lem:visible-superlevel-bypass}
Let $abc$ be a visible path and suppose
\begin{equation*}
 f(a)>f(b),
 \qquad
 f(c)>f(b).
\end{equation*}
Then $a$ and $c$ are joined by a visible path every vertex $w$ of which
satisfies
\begin{equation*}
 f(w)>f(b).
\end{equation*}
\end{lemma}

\begin{proof}
Put $t=f(b)$.  If $ac$ is visible, the edge $ac$ is the required path.  If
$ac$ is collapsed, write $f(a)=f(c)=h>t$.  Then
\Cref{lem:adjacent-edge} and the common neighbour $b$ with lower $f$-value produce another
common neighbour $z$ with $f(z)>h$.  Both $az$ and $zc$ are visible, so
$azc$ is a desired path.

Suppose now that $a$ and $c$ are nonadjacent in $G$.  The \eqref{eq:midpoint-scalar} gives a common neighbour $w$ with
\begin{equation*}
 f(w)>\frac{f(a)+f(c)}2>t.
\end{equation*}
If both $aw$ and $cw$ are visible, use $awc$.  If both are collapsed, then
$f(a)=f(w)=f(c)=h$; the \eqref{eq:midpoint-scalar} and the common neighbour $b$
produce a common neighbour $z$ with $f(z)>h$, and $azc$ works.

In the remaining case exactly one incident edge is collapsed.  After
interchanging $a,c$, assume $aw$ is collapsed and $cw$ is visible.  The
visible path $abcw$ has all values at least $t$ and its collapsed-edge
endpoints $a,w$ have common value greater than $t$.  By
\Cref{lem:collapsed-superlevel-detour}, there is a visible path joining $a$ and $w$
entirely in $\{f>t\}$.  Appending the visible edge $wc$ gives the required
path.
\end{proof}

\begin{theorem}[Unique extrema of a generic $2$-eigenfunction]
\label{thm:unique-maximum}
Let $f\in\Etwo$ satisfy \eqref{eq:visible-strict}, and let $H$ be a connected
component of $\Gvis$.  Then $f|_H$ has a unique maximum point and a unique
minimum point.
\end{theorem}

\begin{proof}
Suppose that two distinct vertices $p,q\in H$ both maximize $f|_H$.  Among
all visible paths $P$ from $p$ to $q$, choose one for which
\begin{equation}\label{eq:path-floor}
 \mu(P):=\min_{x\in V(P)}f(x)
\end{equation}
is as large as possible.  Subject to this condition, choose $P$ so that the
number of vertices $x\in V(P)$ satisfying $f(x)=\mu(P)$ is as small as
possible.

First, we have $pq$ collapsed. Hence $P$ has an internal vertex and
$\mu(P)<f(p)$.  Choose an internal vertex $b$ of $P$ with
$f(b)=\mu(P)$, and denote its two neighbours on $P$ by $a$ and $c$.  Since
$ab$ and $bc$ are visible and no value on $P$ is below $f(b)$,
\begin{equation*}
 f(a)>f(b),
 \qquad
 f(c)>f(b).
\end{equation*}
By \Cref{lem:visible-superlevel-bypass}, there is a visible path from $a$ to $c$ all
of whose vertices have value strictly larger than $f(b)$.  Replace the
segment $abc$ of $P$ by this path and delete closed subwalks, preserving the
endpoints $p,q$.  The resulting simple visible path $P'$ satisfies
$\mu(P')\ge\mu(P)$.  If the inequality is strict, it contradicts the maximal
choice of \eqref{eq:path-floor}.  If equality holds, the vertex $b$ has been
removed and no new vertex of value $\mu(P)$ has been inserted, contradicting
the secondary minimal choice.  Thus the maximum is unique.

The function $-f$ also belongs to $\Etwo$ and satisfies
\eqref{eq:visible-strict}.  Applying the already proved maximum statement to
$-f$ shows that the minimum of $f|_H$ is unique.
\end{proof}


\section{The 1-ball structure at the extrema of a generic 2-eigenfunction}
For the remainder of the section, fix a connected component $H$ of $\Gvis$ and
a generic $2$-eigenfunction $f$ satisfying \eqref{eq:visible-strict}.  By \Cref{thm:unique-maximum}, there is a
unique maximum point $p\in H$. Put $ P=\Phi(p)$. Define
\begin{equation*}
 X=S_1^H(p),\qquad D=|X|,\qquad
 Z=S_1^G(p)\setminus X,\qquad c=|Z|.
\end{equation*}
Every edge $pz$ with $z\in Z$ is collapsed.  For $x\in X$, put
\begin{equation*}
 a_x=f(p)-f(x)>0.
\end{equation*}
In this section, we want to characterize the $1$-ball structure of $G$ at $p$ and do some preparations for the proof of \Cref{thm:intrinsic-Gamma}.

\subsection{1-sphere structure}
\begin{lemma}
\label{lem:no-XZ}
There is no edge of $G$ joining a vertex of $X$ to a vertex of $Z$.
\end{lemma}

\begin{proof}
Let $x\in X$ and $z\in Z$.  Since $pz$ is collapsed,
$\Phi(z)=P$ and $f(z)=f(p)$.  If $xz$ were visible, then $z$ would belong to
the visible component $H$, contradicting the uniqueness of the maximum $p$.
If $xz$ were collapsed, then $f(x)=f(z)=f(p)$, contradicting the visibility of
$px$ and the maximality of $p$.  Hence $xz\notin E(G)$.
\end{proof}

\begin{lemma}
\label{lem:Z-matrix}
Let $B$ be a real symmetric matrix with nonpositive off-diagonal entries.  If
$b>0$ and $Bb=\lambda b$, then $B\succeq\lambda I$.
\end{lemma}

\begin{proof}
Put $C=B-\lambda I$ and $w_{ij}=-C_{ij}\ge0$ for $i\ne j$.  Since $Cb=0$,
for every $u$ one has
\begin{equation*}
 u^TCu
 =\sum_{i<j}w_{ij}b_ib_j
 \left(\frac{u_i}{b_i}-\frac{u_j}{b_j}\right)^2\ge0.
\end{equation*}
Thus $C\succeq0$.
\end{proof}

Let $K=G[X]$ and define
\begin{equation*}
    M=\{q\in S_2^G(p):\ S_1^G(q)\cap X\neq \varnothing,\ S_1^G(q)\cap Z\neq \varnothing \}.
\end{equation*}
For any $x\in X$ define
\begin{equation*}
    t_x=\deg_K(x),\qquad r_x=|S_1^G(x)\setminus(\{p\}\cup X\cup Z)|.
\end{equation*}

The next proposition describes the $1$-sphere structure of $H$ at vertex $p$.
\begin{proposition}
\label{prop:mixed-saturation}
The following statements hold
\begin{enumerate}
    \item \label{en:1Sphere1}For every $x\in X$,
    \begin{equation}
    \label{eq:r_x}
        r_x=D+c-1
    \end{equation}
    and every edge of $K$ is collapsed;
    \item \label{en:1Sphere2}For every $(x,z)\in X\times Z$, there exists a unique vertex $m(x,z)\in M$ adjacent to both $x$ and $z$;
    \item \label{en:1Sphere3}Conversely, every $m\in M$ has exactly one neighbour in $X$ and one neighbour in $Z$. Its
edge to the $X$-neighbour is collapsed.
\end{enumerate}
\end{proposition}

\begin{proof}
By \Cref{lem:no-XZ},
\begin{equation}\label{eq:degree-x}
 \deg_G(x)=1+t_x+r_x,
 \qquad
 S_1^G(p)\cap S_1^G(x)=S_1^K(x).
\end{equation}
Applying \Cref{lem:adjacent-edge} to $px$ and to the eigenfunction $f$, one obtains
\begin{equation*}
 (1+r_x+2t_x-D-c)a_x
 =2\sum_{y\sim_Kx}a_y.
\end{equation*}
Equivalently, with the Laplacian $L_K$,
\begin{equation}\label{eq:B-a}
 \left(I+\operatorname{diag}(r_x-c)-2L_K\right)a=Da.
\end{equation}
The matrix in \eqref{eq:B-a} is symmetric with nonpositive off-diagonal
entries.  By \Cref{lem:Z-matrix}, $D$ is its smallest eigenvalue. 
Thus
\begin{equation*}
    \one_X^T  \left(I+\operatorname{diag}(r_x-c)-2L_K\right)\one_X\geq D\card{X}.
\end{equation*}
This gives
\begin{equation}\label{eq:R-lower}
 R_X:=\sum_{x\in X}r_x\ge D(D+c-1).
\end{equation}

For $q\in S_2^G(p)$ put
\begin{equation*}
 \varepsilon_X(q)=|S_1^G(q)\cap X|,
 \quad
 \varepsilon_Z(q)=|S_1^G(q)\cap Z|,
 \quad
 \varepsilon(q)=\varepsilon_X(q)+\varepsilon_Z(q),
\end{equation*}
and define
\begin{equation*}
 W_{XZ}=\sum_{q\in S_2^G(p)}
 \frac{\varepsilon_X(q)\varepsilon_Z(q)}{\varepsilon(q)}.
\end{equation*}
Write $A=A^G_\infty(p)$.  For
$x\in X$, equations \eqref{eq:A-row-sum} and \eqref{eq:degree-x} give
\begin{equation}\label{eq:row-X}
 ((A-2I)\one)_x=\frac{D+c-1-r_x}{2}.
\end{equation}
For $x\in X$ and $z\in Z$, \Cref{lem:no-XZ} and
\eqref{eq:A-offdiag} give
\begin{equation*}
 A_{xz}=1-2\omega_{xz},
 \qquad
 \omega_{xz}=\sum_{\substack{q\in S_2^G(p),\\q\sim x,z}}\frac1{d_p^{-}(q)}.
\end{equation*}
Therefore
\begin{equation}\label{eq:block-XZ}
 \langle\one_X,A\one_Z\rangle=Dc-2W_{XZ}.
\end{equation}
Since $A-2I\succeq0$, equations
\eqref{eq:row-X}--\eqref{eq:block-XZ} yield
\begin{equation}\label{eq:R-upper-W}
 R_X\le D(D-c-1)+4W_{XZ}.
\end{equation}
Combining \eqref{eq:R-lower} and \eqref{eq:R-upper-W},
\begin{equation}\label{eq:W-lower}
 W_{XZ}\ge\frac{Dc}{2}.
\end{equation}

Choose $q\in M$ and fix an $X$-neighbour $x$ of $q$.  The edge $xq$ cannot be
visible.  Otherwise $q\in H$. For a $Z$-neighbour $z$, either $zq$ is visible and
then $z\in H$, or $zq$ is collapsed and then $f(q)=f(z)=f(p)$.  Both
possibilities contradict the uniqueness of the maximum $p$ on $H$.  Thus
\begin{equation}\label{eq:mixed-image-X}
 \Phi(q)=\Phi(x)
 \qquad\text{for every $X$-neighbour $x$ of a vertex $q\in M$}.
\end{equation}
Applying \eqref{eq:midpoint-scalar} to
$p,q$ gives
\begin{equation*}
    \frac{\Phi(p)+\Phi(q)}{2}=\frac{\varepsilon_X(q)\Phi(q)+\varepsilon_Z(q)\Phi(p)}{\card{S_1^G(p)\cap S_1^G(q)}}
\end{equation*}
Thus $ \varepsilon_X(q)=\varepsilon_Z(q)=:m_q
$.
Consequently,
\begin{equation}\label{eq:W-m}
 W_{XZ}=\frac12\sum_{q\in M}m_q.
\end{equation}

Fix $(x,z)\in X\times Z$.  By \Cref{lem:no-XZ}, $d_G(x,z)=2$ via $p$.
No other common neighbour lies in $S_1^G(p)$, again by \Cref{lem:no-XZ}.  Every
common neighbour $q\ne p$ is in $M$ and satisfies $\Phi(q)=\Phi(x)$ by
\eqref{eq:mixed-image-X}.  If $M_{xz}$ denotes the number of such vertices,
\Cref{lem:midpoint} gives
\[
 \frac{\Phi(x)+P}{2}
 =\frac{P+M_{xz}\Phi(x)}{1+M_{xz}}.
\]
Since $\Phi(x)\ne P$, one has $M_{xz}=1$.  Counting pairs $(x,z)$ through the
vertices $M$ gives
\begin{equation}\label{eq:Dc-m2}
 Dc=\sum_{q\in M}m_q^2.
\end{equation}
Equations \eqref{eq:W-m} and \eqref{eq:Dc-m2} imply
$W_{XZ}\le Dc/2$.  Together with \eqref{eq:W-lower}, equality holds, and every
$m_q=1$.  This proves assertions \ref{en:1Sphere2} and \ref{en:1Sphere3}.

Substituting $W_{XZ}=Dc/2$ into \eqref{eq:R-upper-W}, comparison with
\eqref{eq:R-lower} gives
\begin{equation*}
 R_X=D(D+c-1).
\end{equation*}
Thus $\one_X$ also realizes the smallest Rayleigh quotient in
\eqref{eq:B-a}, so it is an eigenvector with eigenvalue $D$.  Since
$L_K\one_X=0$, this gives \ref{en:1Sphere1}.  Substitution back into
\eqref{eq:B-a} yields $L_Ka=0$.  Hence $a$ is constant on every connected
component of $K$.  If $xy\in E(K)$, then $f(x)=f(y)$, and \eqref{eq:visible-strict} forces $xy$ to be collapsed.
\end{proof}

\subsection{Canonical decomposition of $X$}
We decompose $X$ into $r$ classes such that any two vertices in the same class have the same spectral embedding. This decomposition will play an very important role in the proof of \Cref{thm:intrinsic-Gamma}.

Define the following set
\begin{equation*}
 Y=\{q\in S_2^G(p):S_1^G(q)\cap X\ne\varnothing,
                    \ S_1^G(q)\cap Z=\varnothing\}.
\end{equation*}
For $q\in Y$, put
\begin{equation*}
 B(q)=S_1^G(q)\cap X,
 \qquad k_q=|B(q)|.
\end{equation*}

We aim to prove that $S_2^H(p)=Y$ and that $B_2^H(p)$ is isomorphic to a radius-two ball in the $D$-dimensional hypercube $H_D$.
\begin{proposition}
\label{prop:pure-incidence}
Every $x\in X$ is adjacent to exactly $D-1$ vertices of $Y$.  Every unordered
pair of distinct vertices $x,y\in X$ is contained in exactly one $B(q)$.
\end{proposition}

\begin{proof}
By \eqref{eq:r_x}, $x$ has $D+c-1$ neighbours outside
$\{p\}\cup X\cup Z$.  Exactly $c$ of them are in $M$.  Every remaining neighbour lies in $S_2^G(p)$, has the $X$-neighbour
$x$, and has no $Z$-neighbour, hence belongs to $Y$.  This proves the first
assertion.

For distinct $x,y\in X$, let
\begin{equation*}
 n_{xy}=|\{q\in Y:x,y\in B(q)\}|.
\end{equation*}
We first show $n_{xy}\ge1$.  Put
\begin{equation*}
 C_{xy}^X=S_1^G(x)\cap S_1^G(y)\cap X,
 \qquad s_{xy}=|C_{xy}^X|.
\end{equation*}
A vertex in $M$ cannot be adjacent to two distinct vertices of $X$, and no
vertex of $Z$ is adjacent to $x$ or $y$, by \Cref{lem:no-XZ}.  Hence the common neighbours of $x,y$ in $G$ consist of $p$, the vertices of
$C_{xy}^X$, and the $n_{xy}$ vertices in $Y$ whose $1$-spheres contain $x,y$.
Every $u\in C_{xy}^X$ is joined to $x$ and $y$ by edges of $K$, so
\Cref{prop:mixed-saturation} gives
\begin{equation}\label{eq:Cxy-clone}
 \Phi(u)=\Phi(x)=\Phi(y)
 \qquad(u\in C_{xy}^X).
\end{equation}
In particular, $s_{xy}>0$ implies $\Phi(x)=\Phi(y)$.

Put
\begin{equation*}
 c_q=f(p)-f(q),\qquad A_{xy}=\frac{f(x)+f(y)}{2}.
\end{equation*}
We have
\begin{equation}
\label{eq:1-SphereMidPoint}
    (1+n_{xy}+s_{xy})A_{xy}=f(p)+\sum_{u\in C_{xy}^X}f(u)+\sum_{q:x,y\in B(q)}f(q).
\end{equation}
This comes from \eqref{eq:midpoint-scalar} if $xy\not\in E(G)$ or \eqref{eq:collapsed-average} if $xy\in E(G)$ is a collapsed edge.
If $s_{xy}=0$, the \eqref{eq:1-SphereMidPoint} reduces to
\begin{equation}\label{eq:pair-average-reduced}
 (n_{xy}+1)A_{xy}
 =f(p)+\sum_{q:x,y\in B(q)}f(q).
\end{equation}
If $s_{xy}>0$, together with \eqref{eq:Cxy-clone}, the \eqref{eq:1-SphereMidPoint} also reduces to \eqref{eq:pair-average-reduced}. If $n_{xy}=0$, then $A_{xy}=f(p)$. It is impossible since both $f(p)-f(x)$ and $f(p)-f(y)$ are positive.

For every $q\in Y$, applying \eqref{eq:midpoint-scalar} to $p,q$ gives
\begin{equation}\label{eq:pq-weight}
 \sum_{u\in B(q)}a_u=\frac{k_q}{2}c_q.
\end{equation}
Subtracting the right-hand side of \eqref{eq:pair-average-reduced} from
$(n_{xy}+1)f(p)$ yields
\begin{equation}\label{eq:pair-weight}
 \sum_{q:x,y\in B(q)}c_q
 =(n_{xy}+1)\bigl(f(p)-A_{xy}\bigr)
 =\frac{n_{xy}+1}{2}(a_x+a_y).
\end{equation}
Sum \eqref{eq:pair-weight} over all unordered pairs $\{x,y\}$.  Using
\eqref{eq:pq-weight}, the left-hand side becomes
\[
 \sum_{q\in Y}\binom{k_q}{2}c_q
 =\sum_{x\in X}a_x\sum_{y\ne x}n_{xy}.
\]
The right-hand side is
\[
 \frac12\sum_{x\in X}a_x
 \left(\sum_{y\ne x}n_{xy}+D-1\right).
\]
Therefore
\begin{equation*}
 \sum_{x\in X}a_x
 \left(\sum_{y\ne x}n_{xy}-(D-1)\right)=0.
\end{equation*}
Every bracket is nonnegative because every $n_{xy}\ge1$, and every $a_x>0$.
Hence all brackets vanish.  Each is the sum of $D-1$ positive integers minus
$D-1$, so every $n_{xy}=1$.
\end{proof}

For $x\in X$, define
\begin{equation*}
 v_x=P-\Phi(x)\ne0.
\end{equation*}

\begin{lemma}
\label{lem:pure-spectral-identities}
For every $q\in Y$,
\begin{equation}\label{eq:pure-average-vector}
 \Phi(q)=P-\frac2{k_q}\sum_{u\in B(q)}v_u.
\end{equation}
If $x,y\in B(q)$ are distinct, then
\begin{equation}\label{eq:pure-pair-vector}
 \Phi(q)=\Phi(x)+\Phi(y)-P=P-v_x-v_y.
\end{equation}
Consequently, if $k_q\ge3$, then all vectors $v_u$, $u\in B(q)$, are equal,
and in that case
\begin{equation}\label{eq:fan-image}
 \Phi(q)=P-2v_u\qquad(u\in B(q)).
\end{equation}
The same formula holds when $k_q=1$.
\end{lemma}

\begin{proof}
The common neighbours of $p$ and $q$ are exactly the vertices of $B(q)$.
Applying \Cref{lem:midpoint} gives \eqref{eq:pure-average-vector}. Fix distinct $x,y\in B(q)$. Repeating the argument leading to \eqref{eq:pair-average-reduced}
for each function $\phi_i$ and using $n_{xy}=1$, we obtain
\eqref{eq:pure-pair-vector}.

For $k_q\geq 3$, to prove $v_x=v_y$, take $z\in B(q)$ distinct  from $x$ and $y$. Then 
\begin{equation*}
    \Phi(q)=P-v_{z}-v_{x}=P-v_{z}-v_{y}
\end{equation*}
gives $v_x=v_y$. For $k_q=1$, equation
\eqref{eq:pure-average-vector} directly gives the same formula.
\end{proof}

\begin{proposition}
\label{prop:design}
For fixed $x\in X$, let
\begin{align*}
 \mathcal D_x&=\{q\in Y:B(q)=\{x\}\},
\\
 \mathcal F_x&=\{q\in Y:x\in B(q),\ k_q\ge3\}.
\end{align*}
Then
\begin{equation}\label{eq:local-fan-singleton}
 |\mathcal D_x|
 =\sum_{q\in\mathcal F_x}(k_q-2).
\end{equation}
In particular, a singleton block exists if and only if a block of size at
least three exists.
\end{proposition}
\begin{proof}
    By \Cref{prop:pure-incidence}, it follows that
    \begin{equation*}
        \sum_{q:x\in B(q)}(k_q-1)=D-1=\card{\{q:\ x\in B(q)\}}.
    \end{equation*}
    Thus 
    \begin{equation*}
        0=\sum_{q:x\in B(q)}(k_q-2)=-\card{\mathcal{D}_x}+\sum_{q\in\mathcal{F}_x}(k_q-2).
    \end{equation*}
    This gives the equation \eqref{eq:local-fan-singleton}.
\end{proof}

\begin{proposition}
\label{prop:cubic-transfer}
Let $\mathcal{A}$ be the fourth-order tensor defined by \eqref{eq:FourthOrderTensor}. For every $x\in X$,
\begin{equation}\label{eq:A-vx}
 \cA(v_x)=-2v_x^{\otimes3}.
\end{equation}
\end{proposition}

\begin{proof}
By \eqref{eq:fan-image} and \eqref{eq:pure-average-vector}, we have
\begin{equation*}
    (\Phi(q)-\Phi(x))^{\otimes 3}=\begin{cases}
        -v_x^{\otimes 3} & \text{ for } k_q\neq 2 \text{ and } x\in B(q),\\
        -v_y^{\otimes 3} & \text{ for } B(q)=\{x,y\}.
    \end{cases}
\end{equation*}

Fix $x\in X$. The sets $B(q)\setminus\{x\}$, for $x\in B(q)$ and $k_q\ge2$, partition
$X\setminus\{x\}$. So
\begin{equation}\label{eq:T-p}
\begin{aligned}
     T(p)&=-\sum_{y\in X}v_y^{\otimes3}
     =-v_x^{\otimes 3}-\sum_{q\in\mathcal{F}_x}(k_q-1)v_x^{\otimes 3}-\sum_{q:B(q)=\{x,y\}}v_y^{\otimes 3}\\
     T(x)&=v_x^{\otimes 3}+\sum_{q:x\in B(q)}(\Phi(q)-\Phi(x))^{\otimes 3}=v_x^{\otimes 3}-\sum_{\substack{q:k_q\neq 2\\x\in B(q)}}v_x^{\otimes 3}-\sum_{q:B(q)=\{x,y\}}v_y^{\otimes 3}.
\end{aligned}
\end{equation}
Therefore
\begin{equation*}
\begin{aligned}
     T(p)-T(x)&=-2v_x^{\otimes 3}+\card{\mathcal{D}_x}v_x^{\otimes 3}-\sum_{q\in\mathcal{F}_x}(k_q-2)v_x^{\otimes 3}=-2v_x^{\otimes 3}
\end{aligned}
\end{equation*}
The last equality comes from \eqref{eq:local-fan-singleton}. Since $T=\cA\circ\Phi$,
linearity of $\cA$ proves \eqref{eq:A-vx}.
\end{proof}

\begin{theorem}[Canonical decomposition of $X$]
\label{thm:global-axes}
For $x,y\in X$, either $v_x=v_y$ or $v_x\perp v_y$.  Hence there is a
partition
\begin{equation}\label{eq:axis-classes}
 X=I_1\sqcup\cdots\sqcup I_r,
 \qquad
 v_x=v_i\quad(x\in I_i),
 \qquad
 n_i=|I_i|,
\end{equation}
where the distinct vectors $v_1,\ldots,v_r$ are pairwise orthogonal and form a
basis of $\mathbb{R}^{m}$. Every edge $ab$ of $G$ is either collapsed or satisfies
\begin{equation}\label{eq:global-axis-alignment}
 \Phi(b)-\Phi(a)=t v_i
\end{equation}
for a unique $i$ and a nonzero real number $t$.
\end{theorem}
\begin{proof}
For arbitrary $z,w\in\mathbb{R}^{m}$, by $\TensorA{v_x}{v_y}{z}{w}=\TensorA{v_y}{v_x}{z}{w}$ and
\eqref{eq:A-vx}, we have
\begin{equation*}
 \langle v_x,v_y\rangle
 \langle v_x,z\rangle\langle v_x,w\rangle
 =\langle v_x,v_y\rangle
 \langle v_y,z\rangle\langle v_y,w\rangle.
\end{equation*}
If $\langle v_x,v_y\rangle\ne0$, then
$v_xv_x^T=v_yv_y^T$, so $v_y=\pm v_x$. We expand $f$ in the orthonormal basis $\{\phi_i\}$ of $E_2$ and get
\begin{equation*}
    f(s)=\sum_{i=1}^{m}u_i\phi_i(s)=\langle u_f,\Phi(s)\rangle.
\end{equation*}
where $u_f=(u_i)^T$. Then
\begin{equation*}
    f(p)-f(x)=\langle u_f,\Phi(p)-\Phi(x)\rangle=\langle u_f,  v_x\rangle>0
\end{equation*}
for all $x\in X$. This excludes the minus sign. So $v_x=v_y$ or $v_x\perp v_y$.

At $p$,
\begin{equation}\label{eq:Q-p-classes}
 Q(p)=\sum_{i=1}^rn_iv_i^{\otimes 2}.
\end{equation}
By \eqref{eq:frame-global}, $Q(p)$ is positive definite on $\mathbb{R}^{m}$, so the
orthogonal vectors $v_i$ span $\mathbb{R}^{m}$ and form a basis.   

For $i\ne j$, \eqref{eq:A-vx} gives
$\cA(v_i,v_i,v_j,v_j)=0$.  Equation \eqref{eq:A4-edge} becomes
\begin{equation*}
 0=-\sum_{a\sim b}
 \langle v_i,\Phi(b)-\Phi(a)\rangle^2
 \langle v_j,\Phi(b)-\Phi(a)\rangle^2.
\end{equation*}
Every summand is nonpositive, so each $\Phi(b)-\Phi(a)$ has nonzero projection
onto at most one axis.  Since the $v_i$ form a basis, this is exactly
\eqref{eq:global-axis-alignment}.
\end{proof}

\section{Graph bundles arising from Lichnerowicz-sharp graphs}
Let $\hat{G}$ be the graph with vertex set $V(G)$ and edge set
\begin{equation*}
 E(\hat G)=E(G_{\vis})\cup
 \{xy\text{ is collapsed}:x\text{ and }y\text{ lie in distinct connected components of }G_{\vis}\}.
\end{equation*}
Let $G^\prime$ be the graph whose vertex set is the set of connected components
of $G_{\vis}$, with two components adjacent if and only if they are adjacent in
$G$. In this section, we prove the main result \Cref{thm:BundleStructure}.

\subsection{Structures of visible graphs}

In this subsection, we characterize the structure of $H$. For $k\ge0$, put $ L_k=S_k^H(p)
$. For $x\in L_k$, define
\begin{equation*}
 \Par(x)=S_1^H(x)\cap L_{k-1},
 \qquad
 \Ch(x)=S_1^H(x)\cap L_{k+1},
\end{equation*}
where $L_{-1}=\varnothing$.  

\begin{theorem}
\label{thm:intrinsic-Gamma}
Let $G$ be a Lichnerowicz-sharp graph.  Let $H$ be a connected component
of $\Gvis$, let $f\in\Etwo$ satisfy \eqref{eq:visible-strict}, and let $p$ be the
unique maximum of $f|_H$ and
$D=|S_1^H(p)|$. The following assertions hold for every integer
$0\le k\le D$.

\begin{enumerate}

 \item\label{item:mu-vector}
For every $x\in L_k$, there is a unique vector $\mu(x)=(\mu_1(x),\ldots,\mu_r(x))\in\N^r,\ 0\leq \mu_i\leq n_i$ such that
 \begin{equation*}
  \qquad
  \sum_{i=1}^r\mu_i(x)=k,\qquad \Phi(x)=P-\sum_{i=1}^r\mu_i(x)v_i.
 \end{equation*}

 \item\label{item:unit-structure}
There are exactly $\mu_i(x)$ vertices $y$ in $\Par(x)$ with $\Phi(y)-\Phi(x)=v_i$ and $n_i-\mu_i(x)$ vertices $z$ in $\Ch(x)$  with $\Phi(z)-\Phi(x)=-v_i$.  Every other edge of $G$ at
 $x$ is collapsed.  Consequently,
 \begin{equation}\label{eq:parent-child-count}
  |\Par(x)|=k,
  \qquad
  |\Ch(x)|=D-k,
  \qquad
  \deg_H(x)=D.
 \end{equation}
 In particular, $H$ has no visible edge internal to $L_k$.

 \item\label{item:balance}
 For distinct $x,y\in L_k$,
 \begin{equation}\label{eq:upper-lower-balance}
  |\Par(x)\cap\Par(y)|
  =|\Ch(x)\cap\Ch(y)|\in\{0,1\}.
 \end{equation}
\end{enumerate}
\end{theorem}

Before the proof of \Cref{thm:intrinsic-Gamma}, we first prove two local lemmas.

\begin{lemma}
\label{lem:translated-balance}
Suppose assertions \ref{item:mu-vector} and
\ref{item:unit-structure} of
\Cref{thm:intrinsic-Gamma} hold on $L_k$.  For two distinct
vertices $x,y\in L_k$, 
\begin{equation}\label{eq:translated-balance}
 |\Par(x)\cap\Par(y)|=|\Ch(x)\cap\Ch(y)|.
\end{equation}
\end{lemma}

\begin{proof}
If $|\Par(x)\cap\Par(y)|=|\Ch(x)\cap\Ch(y)|=0$, there is nothing to prove. First  assume that $\card{\Par(x)\cap \Par(y)}\neq 0$. Take $z\in \Par(x)\cap \Par(y)$. Then \ref{item:unit-structure} of \Cref{thm:intrinsic-Gamma} gives
\begin{equation*}
    \Phi(z)=\Phi(x)+v_i=\Phi(y)+v_j,
\end{equation*}
for some $1\leq i,j\leq r$. Thus
\begin{equation*}
    \Phi(y)=\Phi(x)+v_i-v_j.
\end{equation*}

Next, we claim for any $z\in S_1^G(x)\cap S_1^G(y)$, both $xz$ and $yz$ are either visible or collapsed. If $\Phi(x)=\Phi(y)$, this claim obviously holds. If $\Phi(x)\neq \Phi(y)$, since every edge internal to $L_k$ is collapsed, $x$ is not adjacent to $y$, $v_i-v_j\neq 0$ and $z\not\in L_k$. Assume $z\in L_{k-1}$ and $xz$ is visible, but $yz$ is collapsed. Then
\begin{equation*}
    \Phi(z)=\Phi(x)+v_\ell=\Phi(y).
\end{equation*}
Then
\begin{equation*}
    \sum_{i=1}^r\mu_i(y)=\sum_{i=1}^r\mu_i(x)-1=k-1,
\end{equation*}
a contradiction. The same argument proves the claim holds in the case $z\in L_{k+1}$.

Let
\begin{equation*}
\begin{aligned}
        A_k&=\card{\{z\in \Par(x)\cap \Par(y):\ \Phi(z)-\Phi(x)=v_k\}},\\
        B_k&=\card{\{z\in \Ch(x)\cap \Ch(y):\ \Phi(z)-\Phi(x)=-v_k\}}.
\end{aligned}
\end{equation*}
We have
\begin{equation*}
\begin{aligned}
      0&=\sum_{\substack{z\sim x,\\z\sim y}}(2\Phi(z)-\Phi(x)-\Phi(y))\\
      &=\sum_{k=1}^rA_k(2v_k-v_i+v_j)+\sum_{k=1}^rB_k(-2v_k-v_i+v_j)\\
      &=2\sum_{k=1}^r(A_k-B_k)v_k-\sum_{k=1}^r(A_k+B_k)(v_i-v_j),
\end{aligned}
\end{equation*}
which comes from \eqref{eq:midpoint-scalar} if $d_G(x,y)=2$ or \eqref{eq:collapsed-average} if $xy$ is a collapsed edge. Thus $A_k=B_k$ for $k\neq i,j$. If $i\neq j$, we have 
\begin{equation*}
    A_i-B_i=\frac{1}{2}\sum_{k=1}^r(A_k+B_k),\quad A_j-B_j=-\frac{1}{2}\sum_{k=1}^r(A_k+B_k).
\end{equation*}
If $i=j$, we directly have $A_i=B_i$.
Thus
\begin{equation*}
    \card{\Par(x)\cap\Par(y)}-\card{\Ch(x)\cap\Ch(y)}=\sum_{k=1}^r(A_k-B_k)=0.
\end{equation*}
Now assume
\(\Ch(x)\cap\Ch(y)\ne\varnothing\). Choosing a common
child and reversing the signs in the preceding argument gives the same
conclusion.
\end{proof}

\begin{lemma}
\label{lem:residual-bound}
Assume the conclusions of \Cref{thm:intrinsic-Gamma} through layer $L_k$, and
let $w\in L_{k+1}$.  Let $u$ be a visible neighbour of $w$ which is not in $\Par(w)$.  If $ \Phi(u)-\Phi(w)=t v_i$, then $t\geq -1$.
\end{lemma}

\begin{proof}
Choose a parent $x\in\Par(w)$. Put $\Phi(x)=\Phi(w)+v_j$. 

When $xz$ is collapsed, $\Phi(x)=\Phi(z)$. 
When $xz$ is visible, $\Phi(z)-\Phi(x)=v_{j(z)}$ for $z\in L_{k-1}$ and $\Phi(z)-\Phi(x)=-v_{j(z)}$ for $z\in L_{k+1}$. Thus
\begin{equation*}
    \Phi(z)=\Phi(x)+s_zv_{j(z)},
\end{equation*}
where $s_z\in \{-1,0,1\}$. 

Assume $x\sim u$. Since $u$ is adjacent to $w$, $u\not\in L_{k-1}$ and $s_u\neq 1$.  We have
\begin{equation*}
    \Phi(u)-\Phi(w)=v_j+s_uv_{j(u)}=tv_i.
\end{equation*}
Thus $t=1$ if $s_u=0$ or $t=0$ if $s_u=-1$.

Assume $x$ is not adjacent to $u$. Then $d_G(x,u)=2$.  Put $C_{xu}=\card{S_1^G(x)\cap S_1^G(u)}$ and
\begin{equation*}
        \begin{aligned}
                    M_k&=\card{\{z\in S_1^G(x)\cap S_1^G(u):\ \Phi(z)-\Phi(x)=-v_k\}},\\
                P_k&=\card{\{z\in S_1^G(x)\cap S_1^G(u):\ \Phi(z)-\Phi(x)=v_k\}}.
        \end{aligned}
\end{equation*}
Applying \eqref{eq:midpoint-scalar} to $x,u$ gives
\begin{equation*}
    \begin{aligned}
        0&=\sum_{\substack{z\sim u,\\z\sim x}}(2\Phi(z)-\Phi(x)-\Phi(u))\\
        &=\sum_{\substack{z\sim u,\\z\sim x}}(v_j-tv_i+2s_zv_{j(z)})\\
        &=C_{xu}(v_j-tv_i)-2\sum_{k=1}^r(M_k-P_k)v_k.
    \end{aligned}
\end{equation*}
Thus $M_k=P_k$ for $k\neq i,j$. If $i\neq j$ 
\begin{equation*}
    C_{xu}(v_j-tv_i)-2(M_i-P_i)v_i-2(M_j-P_j)v_j=0.
\end{equation*}
Thus we have
\begin{equation*}
     t=-\frac{2(M_i-P_i)}{C_{xu}},\qquad M_j-P_j=\frac{C_{xu}}{2}.   
\end{equation*}
Then $M_{i}\leq C_{xu}-M_j\leq C_{xu}/2$. This gives $t\geq -1$. If $i=j$, then 
\begin{equation*}
        C_{xu}(1-t)v_i-2(M_i-P_i)v_i=0.
\end{equation*}
So $t=1-\frac{2(M_i-P_i)}{C_{xu}}\geq -1$. 

\end{proof}
\begin{proof}[Proof of \Cref{thm:intrinsic-Gamma}]
We prove the assertions by induction on $k$. For $L_0=\{p\}$, set $\mu(p)=0$. There exist exactly $n_i$ neighbours $x$ of $p$ with $\Phi(x)-\Phi(p)=-v_i$ and $\sum_i n_i=D$.  Thus all assertions at $k=0$ are immediate.

Fix $0\le k<D$ and assume all assertions have been proved through $L_k$.  We
prove them for $L_{k+1}$.

\smallskip

\smallskip
\noindent\textbf{Step 1:} We first prove \ref{item:mu-vector}.
Choose a parent $x\in\Par(w)$.  Then for
some $i$,
\[
 \Phi(w)=\Phi(x)-v_i.
\]
Define
\begin{equation}\label{eq:mu-w-definition}
 \mu(w)=\mu(x)+e_i.
\end{equation}
Because the $v_i$ form a basis, the vector in
\eqref{eq:mu-w-definition} is independent of the chosen parent.  Hence
\begin{equation}\label{eq:mu-w-sum}
 \Phi(w)=P-\sum_i\mu_i(w)v_i,
 \qquad
 \sum_i\mu_i(w)=k+1.
\end{equation}
Let $p_i(w)$ be the number of parents $x$ of $w$ with $\Phi(x)-\Phi(w)=v_i$.
Then the number $h_w$ of parents of $w$ is 
\begin{equation}\label{eq:h-sum-p}
 h_w=\sum_i p_i(w).
\end{equation}

\smallskip
\noindent\textbf{Step 2:} we prove $h_w= k+1$. By \eqref{eq:global-axis-alignment}, for any visible edge $\ell w$,
\begin{equation*}
    \Phi(\ell)-\Phi(w)=t_\ell v_i
\end{equation*}
for some $1\leq i\leq r$. Let 
\begin{equation*}
    \mathcal{R}_i(w)=\{\ell\sim_H w:\ \Phi(\ell)-\Phi(w)=t_\ell v_i,\ \ell\not\in L_k\}.
\end{equation*}
By \Cref{lem:residual-bound}, for any $\ell\in \mathcal{R}_i(w)$, $t_\ell\geq -1$. We have
\begin{equation*}
    \begin{aligned}
        Q(w)&=\sum_{s\sim w}(\Phi(s)-\Phi(w))^{\otimes 2}=\sum_{i=1}^r\left(p_i(w)+\sum_{\ell\in \mathcal{R}_i(w)}t^2_\ell\right)v_i^{\otimes 2},\\
        \mathcal{A}(\Phi(w))&=\sum_{s\sim w}(\Phi(s)-\Phi(w))^{\otimes 3}=\sum_{i=1}^r\left(p_i(w)+\sum_{\ell\in \mathcal{R}_i(w)}t^3_\ell\right)v_i^{\otimes 3}.
    \end{aligned}
\end{equation*}
On the other hand, by \Cref{lem:frame-constant} and \eqref{eq:A-vx},
\begin{equation*}
    \begin{aligned}
        Q(w)&=Q(p)=\sum_{i=1}^rn_iv_i^{\otimes 2},\\
        \mathcal{A}(\Phi(w))&=\mathcal{A}(P)-\sum_{i=1}^r\mu_i(w)\mathcal{A}(v_i)=\sum_{i=1}^r(2\mu_i(w)-n_i)v_i^{\otimes 3}.
    \end{aligned}
\end{equation*}
Comparing the coefficients of $v_i^{\otimes 2}$ and $v_i^{\otimes 3}$ gives
\begin{align}
 p_i(w)+\sum_{\ell\in \mathcal{R}_i(w)} t_{\ell}^2&=n_i,
 \label{eq:w-second-moment}\\
 p_i(w)+\sum_{\ell\in \mathcal{R}_i(w)} t_{\ell}^3&=2\mu_i(w)-n_i.
 \notag
\end{align}
Adding the two identities of \eqref{eq:w-second-moment} gives,
\begin{equation}\label{eq:w-moment-sum}
 \sum_{\ell\in \mathcal{R}_i(w)} t_{\ell}^2(1+t_{\ell})
 =2\bigl(\mu_i(w)-p_i(w)\bigr).
\end{equation}
Every term on the left is nonnegative.  Hence
\begin{equation}\label{eq:p-le-mu}
 p_i(w)\le\mu_i(w).
\end{equation}
Summing over $i$ and using \eqref{eq:mu-w-sum}, $h_w\leq k+1$.

Count visible edges between $L_k$ and $L_{k+1}$.  Every $x\in L_k$ has
$D-k$ children, so
\begin{equation}\label{eq:edge-layer-count}
 \sum_{w\in L_{k+1}}h_w=|L_k|(D-k).
\end{equation}
Next, count pairs of vertices in $L_k$ with a common child in $L_{k+1}$.  Since two vertices in $L_k$ have a common child if and only if they have a common parent, thus the number is $\frac12|L_k|k(D-k)$. On another hand, for $w\in L_{k+1}$, the number of pairs of parents of $w$ is $\binom{h_w}{2}$.  Summing over $w\in L_{k+1}$ gives the total number of pairs of vertices in $L_k$ with a common child in $L_{k+1}$.  Hence
\begin{equation}\label{eq:Gamma-edge-count}
 \sum_{w\in L_{k+1}}\binom{h_w}{2}=\frac12|L_k|k(D-k).
\end{equation}
Since $1\le h_w\le k+1$,
\begin{equation}\label{eq:binomial-inequality}
 \binom{h_w}{2}\le\frac{k}{2}h_w.
\end{equation}
After summing, equations \eqref{eq:edge-layer-count} and
\eqref{eq:Gamma-edge-count} show that equality must hold in
\eqref{eq:binomial-inequality} for every $w$.  Since $h_w>0$, its equality
condition is
\begin{equation}\label{eq:h-equality}
 h_w=k+1.
\end{equation}
Equations \eqref{eq:p-le-mu}, \eqref{eq:h-sum-p},
\eqref{eq:mu-w-sum}, and \eqref{eq:h-equality} also imply
\begin{equation}\label{eq:p-equals-mu}
 p_i(w)=\mu_i(w)\qquad(1\le i\le r).
\end{equation}

\smallskip
\noindent\textbf{Step 3: } we finish the proof of \ref{item:unit-structure}.
Substitute \eqref{eq:p-equals-mu} into \eqref{eq:w-moment-sum}.  Its left-hand
side is a sum of nonnegative terms and equals zero.  Hence every nonzero coefficient satisfies
\begin{equation*}
 t_{\ell}=-1.
\end{equation*}
Equation \eqref{eq:w-second-moment} now shows $\card{\mathcal{R}_i(w)}$ is $n_i-\mu_i(w)$. Hence $\mu_i(w)\leq n_i$.

Let $u\in \mathcal{R}_i(w)$. Since $u\sim_Hw$ and $w\in L_{k+1}$, the distance of
$u$ from $p$ is $k$, $k+1$, or $k+2$.  But
\[
 \Phi(u)=P-\sum_{j=1}^r(\mu_j(w)+\delta_{ij})v_j
\]
has coefficient sum $k+2$.  The uniqueness of the representation in the
basis $v_j$ and the inductive spectral formula exclude $L_k$ and $L_{k+1}$.
Therefore $u\in L_{k+2}$.  This proves the numbers of parents and children, $\deg_H(w)=D$, and no visible edges on $L_{k+1}$.

\smallskip
\noindent\textbf{Step 4:} Next, we prove assertions \ref{item:balance}. Suppose distinct $w_1,w_2\in L_{k+1}$ had two distinct common parents
$x,y\in L_k$.  Then $x,y$ have the common child $w_1$.  But $x$ and $y$ have a unique common child, whereas both $w_1$ and $w_2$ are common
children.  This contradiction proves
\[
 |\Par(w_1)\cap\Par(w_2)|\le1.
\]
Then \Cref{lem:translated-balance} gives
\begin{equation}
\label{eq:edge-iff-child}
    \card{\Par(w_1)\cap\Par(w_2)}=\card{\Ch(w_1)\cap\Ch(w_2)}\in \{0,1\}.
\end{equation}
This proves the assertion
\ref{item:balance} on \(L_{k+1}\) and finishes the induction.

Finally, at \(L_D\), \eqref{eq:parent-child-count} gives
\(\Ch(x)=\varnothing\) for every \(x\in L_D\). Hence
\(L_{D+1}=\varnothing\). Since \(H\) is connected, there are no further
distance layers, and therefore
\[
V(H)=\bigsqcup_{j=0}^{D}L_j.
\]
\end{proof}

\begin{corollary}
\label{cor:component-size}
The visible component $H$ is $D$-regular and has $\card{L_k}=\binom{D}{k}$. Thus
\begin{equation}\label{eq:component-size}
 |V(H)|=\sum_{k=0}^D\binom Dk=2^D.
\end{equation}
\end{corollary}

\begin{proof}
The $D$-regularity of $H$ follows from \eqref{eq:parent-child-count}.  For
$0\leq k<D$, count the visible edges between $L_k$ and $L_{k+1}$ in two
ways.  Every vertex of $L_k$ has $D-k$ children, while every vertex of
$L_{k+1}$ has $k+1$ parents.  Hence
\begin{equation}\label{eq:consecutive-layer-count}
    |L_k|(D-k)=|L_{k+1}|(k+1).
\end{equation}
Since $L_0=\{p\}$, induction using
\eqref{eq:consecutive-layer-count} gives
\begin{equation}\label{eq:layer-size}
    |L_k|=\binom{D}{k}\qquad(0\leq k\leq D).
\end{equation}
Since there exist at most $D$ layers,
\begin{equation*}
    |V(H)|=\sum_{k=0}^{D}|L_k|
    =\sum_{k=0}^{D}\binom{D}{k}=2^D,
\end{equation*}
This proves
\eqref{eq:component-size}.
\end{proof}

As an application of \Cref{thm:intrinsic-Gamma}, we give a new rigidity result for hypercubes.
\begin{theorem}[Rigidity of hypercubes]
        \label{thm:hypercube-rigidity}
        Let $G$ be a connected Lichnerowicz-sharp graph. If every edge of $G$ is visible, then $G$ is isomorphic to a hypercube $H_D$ for some $D\geq 1$.
\end{theorem}
\begin{proof}
   Since every edge of $G$ is visible and $G$ is connected, we have
$G_{\vis}=G$. Choose $f\in E_2$ satisfying
\eqref{eq:visible-strict}. By \Cref{thm:unique-maximum}, $f$ has a
unique maximizer $p$. Put $D=\deg_G(p)$. By
\Cref{thm:intrinsic-Gamma}, the graph $G$ is $D$-regular and
$S_D^G(p)\neq\varnothing$. Hence $\diam(G)\geq D$. On the other hand,
\Cref{thm:MyersBound} gives $\diam(G)\leq D$. Therefore
$\diam(G)=D$, and the equality case of \Cref{thm:MyersBound} yields
$G\cong H_D$.
\end{proof}
\begin{example}
    \label{ex:four-dimensional-star-twist}
  Every hypercube satisfies all the assertions in
\Cref{thm:intrinsic-Gamma} and is itself Lichnerowicz-sharp.
Nevertheless, a Lichnerowicz-sharp graph may have a visible component
that is not a hypercube.

    Let
\begin{equation*}
 L_0=\{p\},\qquad
 L_1=\{a_i:i\in[4]\},\qquad
 L_2=\{e_{ij}:1\le i<j\le4\},
\end{equation*}
and let
\begin{equation*}
 L_3=\{b_i:i\in[4]\},\qquad L_4=\{q\}.
\end{equation*}
Define a graph $H$ on $\bigsqcup_{k=0}^4L_k$ by
\begin{align*}
 E(H)={}&\{pa_i:i\in[4]\}
 \cup\{a_ie_{ij},a_je_{ij}:1\le i<j\le4\}\\
 &\cup\{e_{ij}b_i,e_{ij}b_j:1\le i<j\le4\}
 \cup\{b_iq:i\in[4]\}.
\end{align*}
In fact, $H$ is the cospectral mate of $H_4$ by Godsil-McKay switching \cite[Section~1.8.1]{Brouwer2012}.  thus
\begin{equation*}
     \operatorname{Spec}(-L_H)=\{0,2^4,4^6,6^4,8\}.
\end{equation*}

\begin{figure}[htbp]
    \centering

\begin{tikzpicture}[
  x=0.75cm,
  y=0.62cm,
  graph edge/.style={draw=black,line width=0.55pt},
  graph vertex/.style={circle,draw=black,fill=white,line width=0.55pt,
    minimum size=4.8pt,inner sep=0pt}
]
\begin{scope}
  \coordinate (hc-p) at (0,0);
  \coordinate (hc-q) at (6,0);
  \foreach \i/\y in {1/1.5,2/0.5,3/-0.5,4/-1.5}{
    \coordinate (hc-a\i) at (1.35,\y);
  }
  \foreach \i/\y in {4/1.5,3/0.5,2/-0.5,1/-1.5}{
    \coordinate (hc-b\i) at (4.65,\y);
  }
  \foreach \e/\y in {12/2,13/1.2,14/0.4,23/-0.4,24/-1.2,34/-2}{
    \coordinate (hc-e\e) at (3,\y);
  }
  \foreach \i in {1,2,3,4}{
    \draw[graph edge] (hc-p)--(hc-a\i);
    \draw[graph edge] (hc-b\i)--(hc-q);
  }
  \foreach \e/\i/\j in {12/1/2,13/1/3,14/1/4,23/2/3,24/2/4,34/3/4}{
    \draw[graph edge] (hc-a\i)--(hc-e\e);
    \draw[graph edge] (hc-a\j)--(hc-e\e);
  }
  \foreach \e/\i/\j in {12/3/4,13/2/4,14/2/3,23/1/4,24/1/3,34/1/2}{
    \draw[graph edge] (hc-e\e)--(hc-b\i);
    \draw[graph edge] (hc-e\e)--(hc-b\j);
  }
  \foreach \v in {hc-p,hc-a1,hc-a2,hc-a3,hc-a4,
    hc-e12,hc-e13,hc-e14,hc-e23,hc-e24,hc-e34,
    hc-b1,hc-b2,hc-b3,hc-b4,hc-q}{
    \node[graph vertex] at (\v) {};
  }
  \node at (3,-2.7) {$4$-dimensional hypercube $H_4$};
\end{scope}

\begin{scope}[xshift=7.8cm]
  \coordinate (st-p) at (0,0);
  \coordinate (st-q) at (6,0);
  \foreach \i/\y in {1/1.5,2/0.5,3/-0.5,4/-1.5}{
    \coordinate (st-a\i) at (1.35,\y);
    \coordinate (st-b\i) at (4.65,\y);
  }
  \foreach \e/\y in {12/2,13/1.2,14/0.4,23/-0.4,24/-1.2,34/-2}{
    \coordinate (st-e\e) at (3,\y);
  }
  \foreach \i in {1,2,3,4}{
    \draw[graph edge] (st-p)--(st-a\i);
    \draw[graph edge] (st-b\i)--(st-q);
  }
  \foreach \e/\i/\j in {12/1/2,13/1/3,14/1/4,23/2/3,24/2/4,34/3/4}{
    \draw[graph edge] (st-a\i)--(st-e\e);
    \draw[graph edge] (st-a\j)--(st-e\e);
    \draw[graph edge] (st-e\e)--(st-b\i);
    \draw[graph edge] (st-e\e)--(st-b\j);
  }
  \foreach \v in {st-p,st-a1,st-a2,st-a3,st-a4,
    st-e12,st-e13,st-e14,st-e23,st-e24,st-e34,
    st-b1,st-b2,st-b3,st-b4,st-q}{
    \node[graph vertex] at (\v) {};
  }
  \node at (3,-2.7) {$H$};
\end{scope}
\end{tikzpicture}
\end{figure}

Now obtain $G$ from $H$ by adding the following edges inside the layers:
\begin{equation*}
 a_i\sim a_j,\qquad b_i\sim b_j\qquad(i\ne j),
\end{equation*}
and
\begin{equation*}
 e_{ij}\sim e_{k\ell}
 \quad\Longleftrightarrow\quad
 \{i,j\}\cap\{k,\ell\}\ne\varnothing
 \qquad{(e_{ij}\ne e_{k\ell})}.
\end{equation*}
There are three vertex types, namely $L_0\cup L_4$, $L_1\cup L_3$, and
$L_2$.  A direct calculation using \eqref{eq:A-offdiag} and
\eqref{eq:CKLP-diag} gives
\begin{equation*}
 \operatorname{Spec}\bigl(A_\infty^G(x)\bigr)=
 \begin{cases}
  \{2,10^3\},
    &x\in L_0\cup L_4,\\
  \left\{2,4,8,
    \left(10-\dfrac{\sqrt{26}}2\right)^2,
    \left(10+\dfrac{\sqrt{26}}2\right)^2\right\},
    &x\in L_1\cup L_3,\\
  \{2,4,6,8^2,10,14^2\},
    &x\in L_2,
 \end{cases}
\end{equation*}
where exponents indicate multiplicities.  Hence $G$ satisfies
$\CD(2,\infty)$.  Moreover, 
\begin{equation*}
 \operatorname{Spec}(-L_G)=\{0,2,4,6^4,8^4,10^5\}.
\end{equation*}
In particular, $G$ is Lichnerowicz-sharp and $E_2(G)$ is one-dimensional.
It is spanned by the function
\begin{equation*}
 f(x)=2-k\qquad(x\in L_k).
\end{equation*}
Every added edge lies inside a layer and is therefore collapsed by $f$,
whereas $f$ changes by one along every edge of $H$.  Consequently,
$G_{\vis}=H$.

Finally, $a_1$ and $b_1$ have the three common neighbours
$e_{12},e_{13},e_{14}$ in $H$.  In a hypercube, two distinct vertices have
either zero or two common neighbours.  Therefore $H\not\cong H_4$, and the
visible component of the Lichnerowicz-sharp graph $G$ is not a hypercube.

By \Cref{thm:hypercube-rigidity}, we know if $H$ is not isomorphic to a hypercube, it could not be $\CD(2,\infty)$.
\end{example}

\subsection{Bundle structure}
In this subsection, we prove the main theorem \Cref{thm:BundleStructure}. Before the proof, we recall the Perron-Frobenius theorem for an irreducible matrix with non-negative entries.
\begin{theorem}[{\cite[Theorem 2.2.1]{Brouwer2012}}]
\label{thm:PerronFrobenius}
Let $T$ be an  irreducible matrix with non-negative entries. Then there exists a unique positive number $\lambda_0$ such that $\lambda_0$ is the spectral radius of $T$ and every non-negative eigenvector of $T$ corresponds to the eigenvalue $\lambda_0$.
\end{theorem}

\begin{theorem}[Construction of transition map]
\label{thm:component-matching}
Let $C$ and $D$ be distinct connected components of $\Gvis$.  Suppose
that there exists an edge joining $C$ to $D$. Then the edges of $G$ between $C$ and $D$ form a perfect matching. More
precisely, for every $x\in C$ there exists a unique vertex $x'\in D$ such that
$x\sim_Gx'$, and every such edge $xx'$ is collapsed.  The map
\begin{equation*}
 \sigma_{CD}:C\to D,\qquad \sigma_{CD}(x)=x',
\end{equation*}
is an isomorphism $C\cong D$.
\end{theorem}

\begin{proof}
We divide the proof into five steps.

\smallskip
\noindent\textbf{Step 1:} We prove every edge between distinct visible components is
collapsed.
Let $x\in C$ and $y\in D$ satisfy $x\sim_Gy$.  If $xy$ were visible, it would
be an edge of $\Gvis$ and would place $x$ and $y$ in the same connected
component of $\Gvis$.  Since $C\ne D$, this is impossible.  Consequently, every edge joining $C$ to $D$ is collapsed.

\smallskip
\noindent\textbf{Step 2: }
Let $A_C$ and $A_D$ denote the adjacency matrices of $C$ and
$D$, respectively.  Define the $\card{V(C)}\times \card{V(D)}$ matrix $M$ by
\begin{equation*}
 M_{xy}=\begin{cases}
  1,&x\sim_Gy,\\
  0,&x\not\sim_Gy,
 \end{cases}
 \qquad x\in V(C),\ y\in V(D).
\end{equation*}
Since there exists an edge joining $C$ to $D$, $M\ne0$.  We claim that
\begin{equation}\label{eq:intertwining-CD}
 A_CM=MA_D.
\end{equation}

Fix $x\in V(C)$ and $y\in V(D)$, and put
\begin{align}
 a_{xy}&:=(A_CM)_{xy}
 =\bigl|\{z\in C:x\sim_{\Gvis}z,\ z\sim_Gy\}\bigr|,
 \label{eq:a-xy}\\
 b_{xy}&:=(MA_D)_{xy}
 =\bigl|\{z\in D:x\sim_Gz,\ z\sim_{\Gvis}y\}\bigr|.
 \label{eq:b-xy}
\end{align}
We prove $a_{xy}=b_{xy}$.

First suppose that $\Phi(x)=\Phi(y)$.  If $z$ were counted by
\eqref{eq:a-xy}, then $zy$ would join the distinct visible components $C$ and
$D$, hence would be collapsed by $\mathbf{Step 1}$.  Thus
$\Phi(z)=\Phi(y)=\Phi(x)$, contradicting the visibility of $xz$.  Hence
$a_{xy}=0$.  The same argument with $C,D$ exchanged gives $b_{xy}=0$.

Now suppose that $\Phi(x)\ne\Phi(y)$. Then $x$ is not adjacent to $y$ in $G$ by $\mathbf{Step 1}$. If $S_1^G(x)\cap S_1^G(y)=\varnothing$, $a_{xy}=b_{xy}=0$. Now assume $S_1^G(x)\cap S_1^G(y)\neq \varnothing$. Take $z\in S_1^G(x)\cap S_1^G(y)$. Since $x,y$ are in different visible components, one of $xz,yz$ is collapsed and the other is visible. Then \eqref{eq:midpoint-scalar} gives
\begin{equation*}
\begin{aligned}
        0&=\sum_{\substack{z\sim x,\\ z\sim y}}(\Phi(z)-\Phi(x))+\sum_{\substack{z\sim x,\\ z\sim y}}(\Phi(z)-\Phi(y))\\
        &=\sum_{\substack{xz\text{ is visible},\\yz\text{ is collapsed} }}(\Phi(y)-\Phi(x))+\sum_{\substack{xz\text{ is collapsed},\\yz\text{ is visible}}}(\Phi(x)-\Phi(y))\\
        &=(a_{xy}-b_{xy})(\Phi(y)-\Phi(x)).
\end{aligned}
\end{equation*}
Hence $a_{xy}=b_{xy}$.

\smallskip
\noindent\textbf{Step 3:} For $C$, the number of edges joining a vertex of $C$ to $D$ is constant over $C$. The same result also holds for $D$. By \Cref{cor:component-size}, the visible components $C$ and $D$ are regular.
Write their degrees as $d_C$ and $d_D$.  Let $\one_C,\one_D$ denote their
all-one vectors and set $ m=M\one_D\in\R^{V(C)}$. The coordinate $m(x)$ is $|S_1^G(x)\cap D|$.  From
\eqref{eq:intertwining-CD},
\begin{equation}\label{eq:PF-left}
 A_Cm=A_CM\one_D=MA_D\one_D=d_Dm.
\end{equation}
The vector $m$ is nonzero and entrywise nonnegative.  Since $C$ is
connected, $A_C$ is irreducible.  The Perron--Frobenius theorem (\Cref{thm:PerronFrobenius}) applied to
\eqref{eq:PF-left} shows that $d_D$ is the spectral radius of $A_C$.  Since
$A_C$ is the adjacency matrix of a connected $d_C$-regular graph, its spectral
radius is $d_C$.  Hence
\begin{equation}\label{eq:equal-visible-degrees}
 d:=d_C=d_D,
\end{equation}
and the Perron eigenvector is unique up to scale.  Consequently, for some
positive integer $a$,
\begin{equation}\label{eq:constant-row-degree}
 M\one_D=a\one_C.
\end{equation}

Applying the same argument to the transpose of
\eqref{eq:intertwining-CD},
\[
 A_DM^{T}=M^{T}A_C,
\]
gives a positive integer $b$ such that
\begin{equation*}
 M^{T}\one_C=b\one_D.
\end{equation*}
Thus every vertex of $C$ has exactly $a$ neighbours in $D$, and every
vertex of $D$ has exactly $b$ neighbours in $C$.

\smallskip
\noindent\textbf{Step 4:} We show $a=b=1$.
The size conclusion of \Cref{cor:component-size}, applied to both components
and combined with \eqref{eq:equal-visible-degrees}, gives
\begin{equation}\label{eq:equal-component-sizes}
 |C|=2^d=|D|.
\end{equation}
Counting the edges of $G$ between $C$ and $D$ from the two sides yields
\[
 a|C|=b|D|.
\]
Therefore Equation \eqref{eq:equal-component-sizes} implies $a=b$.

Let $p\in C, p^\prime\in D$ be the unique maximizers of a generic $2$-eigenfunction $f$ over $C$ and $D$ respectively. Assume $f(p)\geq f(p^\prime)$. If $a>1$, then there exist at least $a$ vertices of $D$ with $f$-value $f(p)$. Since $f(p)\geq f(p^\prime)$, we have $f(p)=f(p^\prime)$ contradicting the uniqueness of $p^\prime$.

Therefore the edges between $C$ and $D$ form a perfect matching, and
all of them are collapsed by \textbf{Step~1}.  

\smallskip
\noindent\textbf{Step 5:} Let $\sigma_{CD}$ be the map sending every vertex of $C$ to its unique neighbour in $D$. $\sigma_{CD}$ is a graph isomorphism. 
 Fix $x,x^\prime \in C$ with $x\sim_{G_{\vis}} x^\prime$. Let $y=\sigma_{CD}(x), y^\prime=\sigma_{CD}(x^\prime)$. The path $xx^\prime y^\prime$ gives $b_{xy^\prime}=a_{xy^\prime}\neq 0$. Since $y$ is the unique neighbour of $x$ in $D$, to guarantee $b_{xy^\prime}\neq 0$, we must have $y\sim_{G_{\vis}}y^\prime$. Hence $\sigma_{CD}$ is a graph homomorphism. Since $\sigma_{CD}$ is
bijective and both $C$ and $D$ are $d$-regular, it is a graph isomorphism.
\end{proof}

Now, we are ready to prove \Cref{thm:BundleStructure}:
\begin{proof}[Proof of \Cref{thm:BundleStructure}]
Let $\mathcal C$ be the set of connected components of $G_{\vis}$.  Since $G$
is connected, the quotient graph $G^\prime$ on $\mathcal C$ is connected.  If
$C,D\in\mathcal C$ are adjacent, \Cref{thm:component-matching} gives a perfect
matching between them and an isomorphism
\begin{equation*}
 \sigma_{CD}:C\to D,
 \qquad
 \sigma_{DC}=\sigma_{CD}^{-1}.
\end{equation*}
Fix the component $H$.  For each $C\in\mathcal C$, choose an isomorphism
$\tau_C:H\to C$, with $\tau_H$ equal to the identity.  Such isomorphisms exist
by composing the maps $\sigma_{CD}$ along paths in the connected graph
$G^\prime$.  For every oriented edge $(C,D)$ of $G^\prime$, define
\begin{equation*}
 \bar\sigma_{CD}
 =\tau_D^{-1}\circ\sigma_{CD}\circ\tau_C\in\operatorname{Aut}(H).
\end{equation*}
Then $\bar\sigma_{DC}=\bar\sigma_{CD}^{-1}$.  The map
\begin{equation*}
 V(G^\prime)\times V(H)\to V(\hat G),
 \qquad
 (C,u)\to\tau_C(u),
\end{equation*}
is a bijection.  It sends the fiber edges over $C$ to the visible edges of
$C$, and it sends the edge from $(C,u)$ to
$(D,\bar\sigma_{CD}(u))$ to the matching edge from $\tau_C(u)$ to $\sigma_{CD}\tau_C(u)$.
By the definition of $\hat G$, these are all of its edges.  Hence
\begin{equation*}
 \hat G\cong G^\prime\square_{\bar\sigma}H.
\end{equation*}

We next construct a section.  Choose $f\in E_2(G)$ satisfying
\eqref{eq:visible-strict}, and for every $C\in\mathcal C$ let $p_C$ be the
unique maximum of $f|_C$ given by \Cref{thm:unique-maximum}.  If $C\sim D$,
every matching edge $x\sigma_{CD}(x)$ is collapsed, and therefore
\begin{equation*}
 f(\sigma_{CD}(x))=f(x)
 \qquad(x\in C).
\end{equation*}
Thus $\sigma_{CD}(p_C)=p_D$.  Consequently,
\begin{equation*}
 s:V(G^\prime)\to V(\hat G),
 \qquad s(C)=p_C,
\end{equation*}
is a graph homomorphism and the natural projection of $s(C)$ is $C$.  Hence
$s$ is a section.

It remains to prove the assertions about $G^\prime$.  Let
$\pi:V(G)\to V(G^\prime)$ send a vertex to its visible component.  For a
function $g:V(G^\prime)\to\mathbb R$, put $\widetilde g=g\circ\pi$.  Fix
$x\in C$. A direct computation shows
\begin{equation*}
 L_G\widetilde g(x)=L_{G^\prime}g(C),
 \qquad
 \Gamma_G(\widetilde g)(x)=\Gamma_{G^\prime}(g)(C).
\end{equation*}
Both quantities on the left are constant on each visible component.  The definition of
$\Gamma_2$ yields
\begin{equation*}
 \Gamma_{2,G}(\widetilde g)(x)=\Gamma_{2,G^\prime}(g)(C).
\end{equation*}
Since $G$ satisfies $\CD(2,\infty)$, it follows that
$G^\prime$ satisfies $\CD(2,\infty)$ as well.

Finally, suppose that $G^\prime$ is not a singleton.  It is connected, so the
Lichnerowicz bound gives $\lambda_1(G^\prime)\ge2$.  If equality held, there
would be a nonzero function $g$ satisfying
$L_{G^\prime}g=-2g$.  The Laplacian identity above would imply
$\widetilde g\in E_2(G)$.  For adjacent components $C,D$, any matching edge
between them is collapsed, so every function in $E_2(G)$ takes the same value
at its endpoints.  Hence $g(C)=g(D)$.  Connectedness of $G^\prime$ would make
$g$ constant, contradicting $L_{G^\prime}g=-2g$ and $g\ne0$.  Therefore
$\lambda_1(G^\prime)>2$.
\end{proof}

\begin{example}
    \label{ex:bundle-exception}
        \Cref{thm:BundleStructure} cannot, in general, be strengthened to assert that $G$ is a $G[V(H)]$-bundle over $G^\prime$. Let $K_4$ denote the complete graph on $\{1,2,3,4\}$. Consider the graph $G$ obtained from $K_4\square  (K_4\setminus\{14\})$ by deleting the edge joining $(4,2)$ to $(4,3)$. Then $G$ is Lichnerowicz-sharp. However, if $C$ ranges over the connected components of $G_{\vis}$, then the disjoint union of the induced subgraphs $G[V(C)]$
consists of three copies of $K_4\setminus\{14\}$ and one copy of $C_4$.
\begin{figure}[htbp]
    \centering
  \scalebox{0.82}{

\begin{tikzpicture}[
  x=1.15cm,
  y=1.05cm,
  line cap=round,
  line join=round,
  inter-fiber edge/.style={draw=black,line width=0.45pt,densely dashed},
  base edge/.style={draw=black,line width=0.9pt},
  fiber edge/.style={line width=1.25pt},
  fiber vertex/.style={circle,fill=white,line width=0.75pt,
    minimum size=4.2pt,inner sep=0pt},
  base vertex/.style={circle,draw=black,fill=black,
    minimum size=4.6pt,inner sep=0pt}
]

\begin{scope}[shift={(-4.0,-0.5)}]
  \coordinate (f1-1) at (-0.75,-0.65);
  \coordinate (f1-2) at (-0.75, 0.65);
  \coordinate (f1-3) at ( 0.75,-0.20);
  \coordinate (f1-4) at ( 0.75, 1.10);
\end{scope}
\begin{scope}[shift={(-1.4,1.7)}]
  \coordinate (f2-1) at (-0.75,-0.65);
  \coordinate (f2-2) at (-0.75, 0.65);
  \coordinate (f2-3) at ( 0.75,-0.20);
  \coordinate (f2-4) at ( 0.75, 1.10);
\end{scope}
\begin{scope}[shift={(1.5,-1.0)}]
  \coordinate (f3-1) at (-0.75,-0.65);
  \coordinate (f3-2) at (-0.75, 0.65);
  \coordinate (f3-3) at ( 0.75,-0.20);
  \coordinate (f3-4) at ( 0.75, 1.10);
\end{scope}
\begin{scope}[shift={(4.1,1.2)}]
  \coordinate (f4-1) at (-0.75,-0.65);
  \coordinate (f4-2) at (-0.75, 0.65);
  \coordinate (f4-3) at ( 0.75,-0.20);
  \coordinate (f4-4) at ( 0.75, 1.10);
\end{scope}

\foreach \j in {2,3,4}{
  \foreach \a/\b in {1/2,1/3,1/4,2/3,2/4,3/4}{
    \draw[inter-fiber edge] (f\a-\j)--(f\b-\j);
  }
}

\foreach \a/\b in {1/2,1/3,1/4,2/3,2/4,3/4}{
  \draw[base edge] (f\a-1)--(f\b-1);
}

\foreach \a/\b in {1/2,1/3,2/3,2/4,3/4}{
  \draw[fiber edge,draw=red!72!black]   (f1-\a)--(f1-\b);
  \draw[fiber edge,draw=green!55!black] (f2-\a)--(f2-\b);
  \draw[fiber edge,draw=blue!75!black]  (f3-\a)--(f3-\b);
}

\foreach \a/\b in {1/2,1/3,2/4,3/4}{
  \draw[fiber edge,draw=orange!88!black] (f4-\a)--(f4-\b);
}

\foreach \f/\col in {1/red!72!black,2/green!55!black,
                       3/blue!75!black,4/orange!88!black}{
  \foreach \j in {1,2,3,4}{
    \node[fiber vertex,draw=\col] at (f\f-\j) {};
  }
}

\foreach \f in {1,2,3,4}{
  \node[base vertex] at (f\f-1) {};
}

\node[font=\small,text=red!72!black] at (-4.25,1.15)
  {$F_1$};
\node[font=\small,text=green!55!black] at (-1.65,3.45)
  {$F_2$};
\node[font=\small,text=blue!75!black] at (1.5,-2.15)
  {$F_3$};
\node[font=\small,text=orange!88!black] at (4.1,2.80)
  {$F_4$};
\end{tikzpicture}}
  \caption{The figure of $G$ in \Cref{ex:bundle-exception}}
\end{figure}

In fact, the identity map on \(\mathbb R^{V(G)}=\mathbb R^{V(\hat G)}\) induces an isomorphism $E_2(G)\cong E_2(\hat{G})$. Consequently, the \(2\)-eigenspace alone cannot detect the additional edges lying inside the visible components.
\end{example}

The preceding example, together with
\Cref{ex:four-dimensional-star-twist}, illustrates the two difficulties that
arise in general: \(G\) may differ from \(\hat G\), and its visible components
may fail to be hypercubes. The degree condition in \Cref{thm:hypercube-bundle} rules out both phenomena. We now prove this result.
\begin{proof}[Proof of \Cref{thm:hypercube-bundle}]
Let $\mathcal C$ be the set of connected components of $G_{\vis}$.  The
quotient graph $G^\prime$ on $\mathcal C$ is connected.  Moreover,
\Cref{thm:component-matching} shows that adjacent visible components are
isomorphic.  Hence all members of $\mathcal C$ are mutually isomorphic.  By
\Cref{cor:component-size}, they are therefore $d$-regular for one common
integer $d$.  Notice that $d\geq1$: otherwise every edge of the connected
graph $G$ would be collapsed, forcing every function in $E_2(G)$ to be
constant.

Fix $C\in\mathcal C$.  Let $q_C$ be the common value of $\deg_G$ on $C$, set
$r_C=\deg_{G^\prime}(C)$, and let $e_C(x)$ be the number of
neighbours of $x$ that lie in $C$.  By
\Cref{thm:component-matching}, every $x\in C$ has exactly one neighbour in
each visible component adjacent to $C$.  Consequently,
\begin{equation}\label{eq:regular-degree-decomposition}
 q_C=d+r_C+e_C(x).
\end{equation}
In particular, $e_C(x)$ is constant over $C$.

Choose $f\in E_2(G)$ satisfying \eqref{eq:visible-strict}, and let $p_C$ be
the unique maximizer of $f|_C$.  If $p_C$ had a collapsed neighbour $y\in C$,
then $f(y)=f(p_C)$, contradicting the uniqueness of $p_C$.  Thus
$e_C(p_C)=0$, and consequently $e_C\equiv0$ on $C$.  Since $C$ was arbitrary,
there are no collapsed edges whose endpoints lie in the same visible
component.  Therefore
\begin{equation}\label{eq:regular-G-equals-Ghat}
 G=\hat G.
\end{equation}
Equation \eqref{eq:regular-degree-decomposition} now becomes
\begin{equation}\label{eq:component-degree-decomposition}
 q_C=d+r_C.
\end{equation}

We next show that every visible component satisfies $\CD(2,\infty)$.  Fix a
component $C$ and a vertex $x\in C$.   We claim that the
curvature matrix has the block form
\begin{equation}\label{eq:curvature-fiber-base-splitting}
 A_\infty^G(x)=
 \begin{pmatrix}
  A_\infty^C(x)&0\\
  0&B_x
 \end{pmatrix},
\end{equation}
where $B_x$ is the block indexed by neighbours in the other visible components.

Indeed, let $v\sim_C x$, let $D\sim_{G^\prime}C$, and put
\begin{equation*}
 x_D=\sigma_{CD}(x),
 \qquad
 v_D=\sigma_{CD}(v).
\end{equation*}
The vertices $v$ and $x_D$ are not adjacent, because the edges between $C$
and $D$ form the matching induced by $\sigma_{CD}$.  Their common neighbours
are precisely $x$ and $v_D$.  Indeed, the matching property gives the claimed
uniqueness inside $C\cup D$.  A common neighbour in a third visible component
would be joined to both $v$ and $x_D$ by collapsed edges and would therefore
give
\[
 \Phi(v)=\Phi(x_D)=\Phi(x),
\]
contrary to the visibility of $xv$.  Applying the same argument to $x$ and
$v_D$ shows that their only common neighbours are $v$ and $x_D$.  Thus
$v_D\in S_2^G(x)$, $d_x^-(v_D)=2$, and
\begin{equation}
        \label{eq:MixedBlock}
 \varepsilon_{v,x_D}=0,
 \qquad
 \omega_{v,x_D}=\frac12.
\end{equation}
Thus, \eqref{eq:A-offdiag} gives
\begin{equation*}
 (A_\infty^G(x))_{v,x_D}=0.
\end{equation*}
This proves that off-diagonal blocks are $0$.

For two neighbours $v,w\in C$ of $x$, the perfect matching property also
shows that every common neighbour of $v$ and $w$ in $S_2^G(x)$ lies in $C$:
a vertex outside $C$ adjacent to both would have two neighbours in $C$.
Furthermore, if $z\in S_2^C(x)$ and $D\sim_{G^\prime}C$, then $x_D$ is not
adjacent to $z$, again by the matching property.  It follows that
$d_x^-(z)$ has the same value in $G$ and in $C$.  Hence the quantities
$\varepsilon_{vw}$ and $\omega_{vw}$ computed in $G$ agree with those
computed in $C$.  No neighbour $x_D$ of $x$ outside $C$ is adjacent to $v$,
so the corresponding values of $t_v$ also agree.  Finally, each of the $r_C$
neighbours of $x$ outside $C$ contributes $1/2$ to $\Omega_v$ by
\eqref{eq:MixedBlock}.  Therefore
\begin{equation*}
 \Omega_v^G=\Omega_v^C+\frac{r_C}{2}.
\end{equation*}
Since $\deg_G(x)=\deg_G(v)=q_C$ and $C$ is $d$-regular, formulas
\eqref{eq:component-degree-decomposition} and \eqref{eq:CKLP-diag} yield
\begin{align*}
 (A_\infty^G(x))_{vv}
 &=3-q_C+\frac52t_v+2\Omega_v^G\\
 &=3-d+\frac52t_v+2\Omega_v^C
  =(A_\infty^C(x))_{vv}.
\end{align*}
This proves \eqref{eq:curvature-fiber-base-splitting}.

Since $G$ satisfies $\CD(2,\infty)$, we have $A_\infty^G(x)\succeq2I$.
Taking the fiber block in \eqref{eq:curvature-fiber-base-splitting} gives
\begin{equation*}
 A_\infty^C(x)\succeq2I.
\end{equation*}
As $x\in C$ was arbitrary, $C$ satisfies $\CD(2,\infty)$.  

It remains to identify $C$.  In view of
\eqref{eq:regular-G-equals-Ghat}, every neighbour of $x\in C$ outside $C$ is
joined to $x$ by a collapsed edge.  Hence the function $f$ chosen above
satisfies
\begin{equation*}
 L_C(f|_C)=(L_Gf)|_C=-2f|_C.
\end{equation*}
Moreover, $f|_C$ changes along every edge of $C$ by
\eqref{eq:visible-strict}, so it is nonconstant.  Thus $2$ is a nonzero
Laplacian eigenvalue of $C$.  The Lichnerowicz bound for $C$ gives
$\lambda_1(C)\geq2$, and consequently $\lambda_1(C)=2$.  The same function
$f|_C\in E_2(C)$ shows that every edge of $C$ is visible with respect to
$E_2(C)$.  Therefore \Cref{thm:hypercube-rigidity} gives $C\cong H_d$.
Finally, \Cref{thm:BundleStructure} and
\eqref{eq:regular-G-equals-Ghat} show that $G$ is an $H_d$-bundle over
$G^\prime$.
\end{proof}

\section{Spectral properties and non-hypercube fibers}
\label{sec:horizontal-realization}
\subsection{The hypercube spectrum of the visible component}
Although the visible component of a Lichnerowicz-sharp graph is not necessarily a hypercube, the structure described in \Cref{thm:intrinsic-Gamma} can determine the whole adjacency and Laplacian spectrum.  Motivated by the structure, we introduce
a class of layered graphs, called twisted \(D\)-dimensional hypercubes.  We
prove that every twisted \(D\)-dimensional hypercube is adjacency- and
Laplacian-cospectral with \(H_D\) by constructing an orthogonal conjugacy
between its adjacency operator and a diagonal operator with the hypercube
spectrum.
\begin{definition}[Twisted \(D\)-dimensional hypercube]
\label{def:twisted-hypercube}
Let \(H\) be a connected graph. We call \(H\) a twisted
\(D\)-dimensional hypercube if there exists \(p\in V(H)\) such that, with $L_k=S_k^H(p)\ (0\leq k\leq D)$,
\begin{enumerate}
  \item the sets \(L_0,\ldots,L_D\) partition \(V(H)\);
  \item\label{en:TwistedIncidence} every \(x\in L_k\)
has exactly \(k\) neighbors in \(L_{k-1}\), exactly \(D-k\) neighbors
in \(L_{k+1}\), and no neighbors in \(L_k\);
\item for any two
distinct vertices \(x,y\in L_k\),
\[
\bigl|S_1^H(x)\cap S_1^H(y)\cap L_{k-1}\bigr|
=
\bigl|S_1^H(x)\cap S_1^H(y)\cap L_{k+1}\bigr|
\in\{0,1\},
\]
\end{enumerate}
\end{definition}

Let $H$ be a twisted $D$-dimensional hypercube. For $\mathbb{R}^{V(H)}$, we have the following direct sum decomposition
\begin{equation*}
  \mathbb{R}^{V(H)}=\bigoplus_{k=0}^D\mathbb{R}^{L_k}.
\end{equation*}
Let $\one_{x}$ denote the indicator function of the vertex $x\in V(H)$. $\{\one_{x}:x\in V(H)\}$ forms an orthonormal basis of $\mathbb{R}^{V(H)}$. 
Define the following operator by
\begin{equation*}
   \mathrm{U}\one_{x}
 =\sum_{y\in\Ch(x)}\one_{y},\qquad\text{ for every }x\in V(H).
\end{equation*}
Let $\mathrm{D}$ be the adjoint of $\mathrm{U}$. Thus, for $y\in V(H)$,
\begin{equation*}
 \mathrm{D}\one_{y}
 =\sum_{z\in\Par(y)}\one_{z}.
\end{equation*}
Finally, define the operator $\mathrm{H}_0$ by
\begin{equation}\label{eq:weight-operator-spectrum}
 \mathrm{H}_0|_{\mathbb{R}^{L_k}}=(D-2k)I_{\mathbb{R}^{L_k}}.
\end{equation}

\begin{lemma}
\label{lem:up-down-commutator-spectrum}
The operators $\mathrm{D},\mathrm{U}$ and $\mathrm{H}_0$ satisfy
\begin{equation}\label{eq:sl2-first-spectrum}
 [\mathrm{D},\mathrm{U}]=\mathrm{H}_0,
\end{equation}
and
\begin{equation}\label{eq:sl2-other-spectrum}
 [\mathrm{H}_0,\mathrm{D}]=2\mathrm{D},
 \qquad
 [\mathrm{H}_0,\mathrm{U}]=-2\mathrm{U}.
\end{equation}
Moreover, the adjacency operator of $H$ is
\begin{equation}\label{eq:adjacency-up-down-spectrum}
 A_H=\mathrm{U}+\mathrm{D}.
\end{equation}
\end{lemma}

\begin{proof}
Let \(x\in L_k\) and \(y\in L_\ell\).  From the definitions of
\(\mathrm U\) and \(\mathrm D=\mathrm U^*\),
\begin{align*}
 \langle\one_y,\mathrm D\mathrm U\one_x\rangle
 &=|\Ch(x)\cap\Ch(y)|,\\
 \langle\one_y,\mathrm U\mathrm D\one_x\rangle
 &=|\Par(x)\cap\Par(y)|.
\end{align*}
Both quantities vanish when \(k\ne\ell\).  If \(k=\ell\) and \(x\ne y\),
they are equal by \Cref{def:twisted-hypercube}.  For \(x=y\), the same
definition gives
\begin{equation*}
 |\Ch(x)|-|\Par(x)|=(D-k)-k=D-2k.
\end{equation*}
It follows that
\begin{equation*}
 (\mathrm D\mathrm U-\mathrm U\mathrm D)
 \big|_{\mathbb R^{L_k}}=(D-2k)I_{\mathbb R^{L_k}}.
\end{equation*}
Together with \eqref{eq:weight-operator-spectrum}, this proves
\eqref{eq:sl2-first-spectrum}.

Now let \(1\le k\le D\) and \(\xi\in\mathbb R^{L_k}\).  Since
\(\mathrm D\xi\in\mathbb R^{L_{k-1}}\),
\begin{align*}
 (\mathrm H_0\mathrm D-\mathrm D\mathrm H_0)\xi
 &=\bigl(D-2(k-1)-(D-2k)\bigr)\mathrm D\xi\\
 &=2\mathrm D\xi.
\end{align*}
This identity also holds on \(\mathbb R^{L_0}\), because
\(\mathrm D\mathbb R^{L_0}=0\).  Similarly, for \(0\le k\le D-1\) and
\(\xi\in\mathbb R^{L_k}\),
\begin{align*}
 (\mathrm H_0\mathrm U-\mathrm U\mathrm H_0)\xi
 &=\bigl(D-2(k+1)-(D-2k)\bigr)\mathrm U\xi\\
 &=-2\mathrm U\xi.
\end{align*}
At \(k=D\), the same identity follows from
\(\mathrm U\mathbb R^{L_D}=0\).  This proves
\eqref{eq:sl2-other-spectrum}.

Finally, adjacent vertices have distances from \(p\) differing by at most
one, while \Cref{def:twisted-hypercube} excludes edges within a layer.
Therefore, for every \(x\in V(H)\),
\begin{equation*}
 A_H\one_x
 =\sum_{y\sim_H x}\one_y
 =\mathrm U\one_x+\mathrm D\one_x.
\end{equation*}
Since the vectors \(\one_x\), \(x\in V(H)\), form a basis, this proves
\eqref{eq:adjacency-up-down-spectrum}.
\end{proof}

Let us recall the exponential of a linear operator $T$ on a finite-dimensional vector space $V$:
\begin{equation*}
 \exp(T):=\sum_{n=0}^\infty\frac{T^n}{n!}.
\end{equation*}
It is standard that the exponential series converges in finite dimensions.

\begin{lemma}
\label{lem:orthogonal-rotation-spectrum}
Let
\begin{equation}\label{eq:K-operator-spectrum}
 \mathrm{K}:=\frac12(\mathrm{D}-\mathrm{U}).
\end{equation}
Then $\mathrm{K}^*=-\mathrm{K}$, and
\begin{equation}\label{eq:orthogonal-conjugacy-spectrum}
 A_H
 =\exp\!\left(-\frac{\pi}{2}\mathrm{K}\right)
  \mathrm{H}_0
  \exp\!\left(\frac{\pi}{2}\mathrm{K}\right).
\end{equation}
In particular, $A_H$ and $\mathrm{H}_0$ are orthogonally similar.
\end{lemma}

\begin{proof}
Because $\mathrm{D}=\mathrm{U}^*$, equation
\eqref{eq:K-operator-spectrum} gives $\mathrm{K}^*=-\mathrm{K}$.  Thus
\begin{equation*}
 Q_\theta:=\exp(\theta\mathrm{K})
\end{equation*}
is orthogonal for every real $\theta$.

Put $\mathrm{A}:=\mathrm{U}+\mathrm{D}=A_H$.  From
\eqref{eq:sl2-first-spectrum}--\eqref{eq:sl2-other-spectrum},
\begin{align}
 [\mathrm{K},\mathrm{H}_0]
 &=\frac12\bigl([\mathrm{D},\mathrm{H}_0]-[\mathrm{U},\mathrm{H}_0]\bigr)
 =-\mathrm{A},
 \label{eq:K-H-commutator-spectrum}\\
 [\mathrm{K},\mathrm{A}]
 &=\frac12[\mathrm{D}-\mathrm{U},\mathrm{D}+\mathrm{U}]
 =[\mathrm{D},\mathrm{U}]
 =\mathrm{H}_0.
 \label{eq:K-A-commutator-spectrum}
\end{align}
Define
\begin{equation*}
 F(\theta):=Q_\theta\mathrm{H}_0Q_\theta^{-1}.
\end{equation*}
Then by \eqref{eq:K-H-commutator-spectrum}--\eqref{eq:K-A-commutator-spectrum},  $\ddot{F}(\theta)=-F(\theta)$. The solution of this second-order linear differential equation is
\begin{equation}\label{eq:F-solution-spectrum}
 F(\theta)=\mathrm{H}_0\cos\theta-\mathrm{A}\sin\theta.
\end{equation}
Substituting $\theta=-\pi/2$ gives
\eqref{eq:orthogonal-conjugacy-spectrum}.
\end{proof}

\begin{theorem}
\label{thm:hypercube-cospectrality}
Let $H$ be a twisted $D$-dimensional hypercube.  Its adjacency
spectrum is
\begin{equation}\label{eq:adjacency-spectrum-visible}
 \operatorname{Spec}(A_H)
 =\left\{
   (D-2k)^{\binom Dk}:0\le k\le D
  \right\}.
\end{equation}
Consequently,  for the Laplacian $L_H=A_H-DI$,
\begin{equation}\label{eq:negative-laplacian-spectrum-visible}
 \operatorname{Spec}(L_H)
 =\left\{
   (-2k)^{\binom Dk}:0\le k\le D
  \right\}.
\end{equation}
Thus $H$ is adjacency- and Laplacian-cospectral with the $D$-dimensional
hypercube $H_D$.
\end{theorem}

\begin{proof}
The $D$-regularity of $H$ is from \ref{en:TwistedIncidence}. Moreover, the same double-counting argument as in the proof of
\Cref{cor:component-size} applies here: counting the edges between \(L_k\) and
\(L_{k+1}\) from the two sides gives
\begin{equation*}
 |L_k|(D-k)=|L_{k+1}|(k+1)
 \qquad(0\le k<D).
\end{equation*}
Since \(L_0=\{p\}\), it follows inductively that
\begin{equation*}
 |L_k|=\binom Dk
 \qquad(0\le k\le D).
\end{equation*}

By \Cref{lem:orthogonal-rotation-spectrum}, \(A_H\) is orthogonally similar to
\(\mathrm H_0\).  Since \(\mathrm H_0\) acts on \(\mathbb R^{L_k}\) as
multiplication by \(D-2k\), the preceding layer sizes give
\eqref{eq:adjacency-spectrum-visible}.  Because \(H\) is \(D\)-regular,
\(L_H=A_H-DI\), so \eqref{eq:negative-laplacian-spectrum-visible} follows by
shifting each adjacency eigenvalue by \(-D\).  The adjacency and Laplacian of $D$-dimensional hypercube $H_D$ is exactly \Cref{eq:adjacency-spectrum-visible} and \Cref{eq:negative-laplacian-spectrum-visible} \cite[Section 1.4.6]{Brouwer2012}.
\end{proof}

\subsection{Non-hypercube fibers}
\label{sec:NonHypercubeFibers}
In this subsection, we want to show that there exist Lichnerowicz-sharp graphs whose visible components are $D$-regular and not isomorphic to the hypercube $H_D$ for any $D\geq 4$.

First, we generalize the graph $H$ in \Cref{ex:four-dimensional-star-twist} to every even degree.  Let $D=2m$ with $m\geq1$, put
$[D]=\{1,\ldots,D\}$, and set
\begin{equation}\label{eq:star-twisted-layers}
 L_k=\left\{(k,A):A\in\binom{[D]}{\min\{k,D-k\}}\right\},
 \qquad 0\leq k\leq D.
\end{equation}
The first coordinate in $(k,A)$ distinguishes vertices in $L_k$ and $L_{D-k}$.
Define a graph $H_D^\star$ on $\bigsqcup_{k=0}^D L_k$ by joining $L_k$ to $L_{k+1}$ according to
\begin{equation}\label{eq:star-twisted-incidence}
 (k,A)\sim(k+1,B)
 \quad\text{ if and only if}\quad
 \begin{cases}
  A\subset B,&0\leq k<m,\\
  B\subset A,&m\leq k<D.
 \end{cases}
\end{equation}
Thus the two truncated Boolean lattices are glued to the middle layer by the
same incidence relation.  This is the twist: the hypercube gluing would use
the complementary incidence relation on one side.

We obtain $G_D^\star$ from $H_D^\star$ by adding the edges in every
layer.  More precisely, for distinct vertices in the same layer, set
\begin{equation}\label{eq:star-twisted-horizontal-edges}
 (k,A)\sim_{G_D^\star}(k,B)
 \quad\text{ if and only if }\quad
 |A\mathbin{\triangle}B|=2.
\end{equation}
Consequently,
\begin{equation}\label{eq:star-twisted-degrees}
 \deg_{H_D^\star}(x)=D,
 \qquad
 \deg_{G_D^\star}(x)=D+k(D-k)
 \qquad(x\in L_k).
\end{equation}
For $D=4$, the identification
\begin{equation*}
 (0,\varnothing)=p,\quad (1,\{i\})=a_i,\quad
 (2,\{i,j\})=e_{ij},\quad (3,\{i\})=b_i,\quad
 (4,\varnothing)=q
\end{equation*}
recovers exactly the graph in \Cref{ex:four-dimensional-star-twist}.

\begin{proposition}\label{prop:even-dimensional-star-twist}
For every even $D\geq2$, the graph $G_D^\star$ satisfies
$\CD(2,\infty)$ and
\begin{equation*}
 \lambda_1(G_D^\star)=2,
 \qquad
 E_2(G_D^\star)=\operatorname{span}\{f\},
 \qquad
 f(k,A)=m-k.
\end{equation*}
In particular, $G_D^\star$ is Lichnerowicz-sharp and
\begin{equation*}
 (G_D^\star)_{\vis}=H_D^\star.
\end{equation*}
If $D\geq4$, then $H_D^\star$ is not isomorphic to the hypercube $H_D$.
\end{proposition}

\begin{proof}
We first verify the curvature condition.  The involution
\begin{equation*}
 (k,A)\to(D-k,A)
\end{equation*}
is an automorphism of $G_D^\star$, so it is enough to consider the layers
$L_0,\ldots,L_m$.

Let $x=(0,\varnothing)$.  Its $D$ neighbours form a clique, and two distinct
neighbours have $\omega_{ij}=1/2$.  Hence
\begin{equation}\label{eq:star-twisted-endpoint-curvature}
 A_\infty(x)=(2D+2)I_D-2J_D,
 \qquad
 \operatorname{Spec}\bigl(A_\infty(x)\bigr)
 =\{2,(2D+2)^{D-1}\}.
\end{equation}

Next, fix $x=(r,A)\in L_r$, where $1\leq r<m$, and put $s=D-r$.
The neighbours of $x$ are naturally indexed as
\begin{equation*}
 P_i=(r-1,A\setminus\{i\}),\qquad
 C_j=(r+1,A\cup\{j\}),\qquad
 W_{ij}=(r,(A\setminus\{i\})\cup\{j\}),
\end{equation*}
where $i\in A$ and $j\in[D]\setminus A$.  Thus the three types have sizes
$r,s$, and $rs$, respectively.  Define
\begin{equation*}
 \eta_r=\begin{cases}
  \dfrac12,&1\leq r\leq m-2,\\[1mm]
  \dfrac1s,&r=m-1.
 \end{cases}
\end{equation*}
A count in $B_2(x)$ gives the following values from
\Cref{def:curMatrix}; the two vertices in each row are distinct, and any
relations among their indices are stated explicitly.
\begin{equation}\label{eq:star-twisted-omega-table}
\begin{array}{c|c|c}
 y,z&\varepsilon_{yz}&\omega_{yz}\\ \hline
 P_i,P_h&1&(s+2)/4\\
 C_j,C_\ell&1&r/4+\eta_r\\
 P_i,C_j&0&0\\
 P_i,W_{ij}\ &1&(r-1)/4\\
 P_i,W_{hj},\ i\ne h&0&1/4\\
 C_j,W_{ij}\ &1&(s-1)/4\\
 C_j,W_{i\ell},\ j\ne\ell&0&1/4\\
 W_{ij},W_{i\ell}&1&r/4\\
 W_{ij},W_{hj}&1&s/4\\
 W_{ij},W_{h\ell},\ i\ne h,\ j\ne\ell&0&1/4
\end{array}
\end{equation}
The only exceptional entry is the second row when $r=m-1$.  Across the
middle-layer, the relevant common second-sphere vertex has
$d_x^-=s$, so its contribution is $1/s$ instead of the usual $1/2$.
Furthermore,
\begin{align*}
 \deg(x)&=D+rs,&
 \deg(P_i)&=D+(r-1)(s+1),\\
 \deg(C_j)&=D+(r+1)(s-1),&
 \deg(W_{ij})&=D+rs
\end{align*}
and
\begin{equation*}
\begin{array}{rcl@{\qquad}rcl@{\qquad}rcl}
 t_{P_i} & = & D-1,
 &t_{C_j} & = & D-1,
 &t_{W_{ij}} & = & D,\\
 \Omega_{P_i} & = & \dfrac{(r-1)(3s+2)}4,
 &\Omega_{C_j} & = & (s-1)\left(\dfrac{3r}{4}+\eta_r\right),
 &\Omega_{W_{ij}} & = & \dfrac{3(rs-1)}4.
\end{array}
\end{equation*}
Thus the curvature matrix $A_\infty(x)$ is
\begin{equation}\label{eq:star-twisted-curvature-table}
\begin{array}{c|c}
 y,z&A_\infty(x)_{yz}\\ \hline
 P_i,P_i&\dfrac{rs+4r+s}{2}\\
 P_i,P_h,\ h\neq i&-2-\dfrac{s}{2}\\
 C_j,C_j&1+\dfrac{rs+r+2s}{2}+2(s-1)\eta_r\\
 C_j,C_\ell,\ \ell\neq j&-1-\dfrac{r}{2}-2\eta_r\\
 P_i,C_j&1\\
 P_i,W_{ij}&-\dfrac{r+1}{2}\\
 P_i,W_{hj},\ i\ne h&\dfrac12\\
 C_j,W_{ij}&-\dfrac{s+1}{2}\\
 C_j,W_{i\ell},\ j\ne\ell&\dfrac12\\
 W_{ij},W_{ij}&\dfrac{rs+3r+3s+3}{2}\\
 W_{ij},W_{i\ell},\ j\neq\ell&-1-\dfrac{r}{2}\\
 W_{ij},W_{hj},\ i\ne h&-1-\dfrac{s}{2}\\
 W_{ij},W_{h\ell},\ i\ne h,\ j\ne\ell&\dfrac12
\end{array}
\end{equation}

For $y\in S_1(x)$, let $\mathbf e_y$ be the corresponding coordinate vector
in $\mathbb R^{S_1(x)}$.  First define the three unit vectors
\begin{equation*}
 p_0=\frac1{\sqrt r}\sum_{i\in A}\mathbf e_{P_i},\qquad
 c_0=\frac1{\sqrt s}\sum_{j\notin A}\mathbf e_{C_j},\qquad
 w_0=\frac1{\sqrt{rs}}\sum_{\substack{i\in A\\j\notin A}}
       \mathbf e_{W_{ij}}.
\end{equation*}
Direct multiplication using \eqref{eq:star-twisted-curvature-table} shows that
$\operatorname{span}\{p_0,c_0,w_0\}$ is invariant under $A_\infty(x)$ and that,
with respect to the ordered orthonormal basis $(p_0,c_0,w_0)$, the restriction
of $A_\infty(x)$ is
\begin{equation}\label{eq:star-twisted-constant-block}
 T_{r,s}=\begin{pmatrix}
  s+2&\sqrt{rs}&-\sqrt{s}\\
  \sqrt{rs}&r+2&-\sqrt r\\
  -\sqrt{s}&-\sqrt r&m+4
 \end{pmatrix},
 \qquad
 \operatorname{Spec}(T_{r,s})=\{2,m+2,D+4\}.
\end{equation}

Next, for any unit vector $a=(a_i)_{i\in A}$ satisfying
$\sum_{i\in A}a_i=0$, set
\begin{equation*}
 p(a)=\sum_{i\in A}a_i\mathbf e_{P_i},\qquad
 w_A(a)=\frac1{\sqrt s}
 \sum_{\substack{i\in A\\j\notin A}}a_i\mathbf e_{W_{ij}}.
\end{equation*}
The vectors $p(a)$ and $w_A(a)$ are orthonormal, their span is invariant under
$A_\infty(x)$, and the restriction of $A_\infty(x)$ to this span, in the
ordered basis $(p(a),w_A(a))$, is
\begin{equation}\label{eq:star-twisted-row-block}
 R_{r,s}=\frac12\begin{pmatrix}
  rs+4r+2s+4&-(r+2)\sqrt{s}\\
  -(r+2)\sqrt{s}&4r+s+8
 \end{pmatrix}.
\end{equation}
It satisfies
\begin{equation}\label{eq:star-twisted-row-determinant}
 \det(R_{r,s}-2I)
 =\frac{3r^2s+16r^2+rs^2+12rs+16r+2s^2+4s}{4}>0.
\end{equation}
Since the first diagonal entry of $R_{r,s}-2I$ is positive,
$R_{r,s}\succ2I$.

Similarly, for any unit vector $b=(b_j)_{j\notin A}$ satisfying
$\sum_{j\notin A}b_j=0$, set
\begin{equation*}
 c(b)=\sum_{j\notin A}b_j\mathbf e_{C_j},\qquad
 w_B(b)=\frac1{\sqrt r}
 \sum_{\substack{i\in A\\j\notin A}}b_j\mathbf e_{W_{ij}}.
\end{equation*}
The span of the orthonormal pair $(c(b),w_B(b))$ is invariant.  If
$r\leq m-2$, the restriction of $A_\infty(x)$ to this span, in that ordered
basis, is $R_{s,r}$ and is therefore strictly larger than $2I$.  If $r=m-1$,
then $s=r+2$ and the restriction is instead
\begin{equation*}
 C_{r,s}^{\partial}=\frac12\begin{pmatrix}
  rs+2r+2s+8&-(s+2)\sqrt r\\
  -(s+2)\sqrt r&r+4s+8
 \end{pmatrix}.
\end{equation*}
In this case,
\begin{equation}
\label{eq:star-twisted-boundary-block}
 \det(C_{r,s}^{\partial}-2I)
 =\frac{(s+2)(r^2+3rs+4r+8s+8)}4>0,
\end{equation}
and the first diagonal entry of $C_{r,s}^{\partial}-2I$ is positive.  Hence
$C_{r,s}^{\partial}\succ2I$.

Finally, let $Z=(z_{ij})_{i\in A,\,j\notin A}$ satisfy
\begin{equation*}
 \sum_{j\notin A}z_{ij}=0\quad(i\in A),
 \qquad
 \sum_{i\in A}z_{ij}=0\quad(j\notin A),
\end{equation*}
and put
\begin{equation*}
 w(Z)=\sum_{\substack{i\in A\\j\notin A}}z_{ij}\mathbf e_{W_{ij}}.
\end{equation*}
Then direct multiplication gives
\begin{equation*}
 A_\infty(x)w(Z)=c_{r,s}w(Z),
 \qquad
 c_{r,s}=\frac{rs+4r+4s+8}{2}>2.
\end{equation*}
Let $\{a^{(\mu)}\}_{\mu=1}^{r-1}$ and
$\{b^{(\nu)}\}_{\nu=1}^{s-1}$ be orthonormal bases of the zero-sum
subspaces
\begin{equation*}
 \mathbb R_0^A
 :=\left\{a=(a_i)_{i\in A}\in\mathbb R^A:
           \sum_{i\in A}a_i=0\right\},
 \qquad
 \mathbb R_0^{[D]\setminus A}
 :=\left\{b=(b_j)_{j\notin A}\in\mathbb R^{[D]\setminus A}:
           \sum_{j\notin A}b_j=0\right\},
\end{equation*}
respectively. Then $\mathbb{R}^{S_1(x)}$ has the following orthogonal decomposition
\begin{equation}
\begin{aligned}
\label{eq:OrthogonalDecomposition}
    \mathbb{R}^{S_1(x)}=&\Span\{p_0,c_0,w_0\}\bigoplus_{\mu=1}^{r-1}\Span\{p(a^{(\mu)}),w_A(a^{(\mu)})\}\bigoplus_{\nu=1}^{s-1}\Span\{c(b^{(\nu)}),w_B(b^{(\nu)})\}\\
    &\bigoplus\left\{w((z_{ij})_{i\in A,j\not\in A}):\ \sum_{j\not\in A}z_{ij}=0,\ \sum_{i\in A}z_{ij}=0\right\}.
\end{aligned}
\end{equation}
These subspaces in \eqref{eq:OrthogonalDecomposition} are invariant under $A_{\infty}(x)$. Together with \eqref{eq:star-twisted-constant-block}, \eqref{eq:star-twisted-row-determinant} and \eqref{eq:star-twisted-boundary-block}, this proves 
\begin{equation*}
 \lambda_{\min}\bigl(A_\infty(x)\bigr)=2
 \qquad(1\leq r<m).
\end{equation*}

It remains to consider a middle-layer vertex $x=(m,A)$.  Put
$B=[D]\setminus A$ and write its neighbours as
\begin{equation*}
 P_i=(m-1,A\setminus\{i\}),\qquad
 Q_i=(m+1,A\setminus\{i\}),\qquad
 W_{ij}=(m,(A\setminus\{i\})\cup\{j\}),
\end{equation*}
where $i\in A$ and $j\in B$. 

A count in $B_2(x)$ gives the following values from
\Cref{def:curMatrix}; the two vertices in each row are distinct, and any
relations among their indices are stated explicitly.
\begin{equation}\label{eq:star-twisted-middle-omega-table}
\begin{array}{c|c|c}
 y,z&\varepsilon_{yz}&\omega_{yz}\\ \hline
 P_i,P_h&1&(m+2)/4\\
 Q_i,Q_h&1&(m+2)/4\\
 P_i,Q_h&0&0\\
 P_i,W_{ij}&1&(m-1)/4\\
 P_i,W_{hj},\ i\ne h&0&1/4\\
 Q_i,W_{ij}&1&(m-1)/4\\
 Q_i,W_{hj},\ i\ne h&0&1/4\\
 W_{ij},W_{i\ell}&1&(m-1)/4\\
 W_{ij},W_{hj}&1&(m+1)/4\\
 W_{ij},W_{h\ell},\ i\ne h,\ j\ne\ell&0&1/4
\end{array}
\end{equation}
Moreover,
\begin{align*}
 \deg(x)&=\deg(W_{ij})=m^2+2m,
 &\deg(P_i)&=\deg(Q_i)=m^2+2m-1,\\
 t_{P_i}&=t_{Q_i}=2m-1,
 &t_{W_{ij}}&=2m,
\end{align*}
and
\begin{equation*}
 \Omega_{P_i}=\Omega_{Q_i}=\frac{(m-1)(3m+2)}4,
 \qquad
 \Omega_{W_{ij}}=\frac{3(m^2-1)}4.
\end{equation*}
Thus the curvature matrix $A_\infty(x)$ is
\begin{equation}\label{eq:star-twisted-middle-curvature-table}
\begin{array}{c|c}
 y,z&A_\infty(x)_{yz}\\ \hline
 P_i,P_i&\dfrac{m^2+5m}{2}\\
 P_i,P_h,\ h\ne i&-2-\dfrac m2\\
 Q_i,Q_i&\dfrac{m^2+5m}{2}\\
 Q_i,Q_h,\ h\ne i&-2-\dfrac m2\\
 P_i,Q_h&1\\
 P_i,W_{ij}&-\dfrac{m+1}{2}\\
 P_i,W_{hj},\ i\ne h&\dfrac12\\
 Q_i,W_{ij}&-\dfrac{m+1}{2}\\
 Q_i,W_{hj},\ i\ne h&\dfrac12\\
 W_{ij},W_{ij}&\dfrac{m^2+6m+3}{2}\\
 W_{ij},W_{i\ell},\ j\ne\ell&-\dfrac{m+1}{2}\\
 W_{ij},W_{hj},\ i\ne h&-\dfrac{m+3}{2}\\
 W_{ij},W_{h\ell},\ i\ne h,\ j\ne\ell&\dfrac12
\end{array}
\end{equation}

Define
\begin{equation*}
 p_0=\frac1{\sqrt m}\sum_{i\in A}\mathbf e_{P_i},\qquad
 q_0=\frac1{\sqrt m}\sum_{i\in A}\mathbf e_{Q_i},\qquad
 w_0=\frac1m\sum_{\substack{i\in A\\j\in B}}\mathbf e_{W_{ij}}.
\end{equation*}
Direct multiplication using \eqref{eq:star-twisted-middle-curvature-table}
shows that $\operatorname{span}\{p_0,q_0,w_0\}$ is invariant under
$A_\infty(x)$, and the restriction of $A_\infty(x)$ to it, with respect to the
ordered orthonormal basis $(p_0,q_0,w_0)$, is
\begin{equation}\label{eq:star-twisted-middle-constant-block}
 T_m=\begin{pmatrix}
  m+2&m&-\sqrt m\\
  m&m+2&-\sqrt m\\
  -\sqrt m&-\sqrt m&m+4
 \end{pmatrix},
 \qquad
 \operatorname{Spec}(T_m)=\{2,m+2,2m+4\}.
\end{equation}

For any unit vector $a=(a_i)_{i\in A}$ with $\sum_{i\in A}a_i=0$,
define
\begin{equation*}
 p(a)=\sum_{i\in A}a_i\mathbf e_{P_i},\qquad
 q(a)=\sum_{i\in A}a_i\mathbf e_{Q_i},\qquad
 w_A(a)=\frac1{\sqrt m}
 \sum_{\substack{i\in A\\j\in B}}a_i\mathbf e_{W_{ij}}.
\end{equation*}
These three vectors are orthonormal, their span is invariant under
$A_\infty(x)$, and the restriction in the ordered basis
$(p(a),q(a),w_A(a))$ is
\begin{equation}\label{eq:star-twisted-middle-first-index-block}
 S_m=\begin{pmatrix}
  \alpha_m&0&\beta_m\\
  0&\alpha_m&\beta_m\\
  \beta_m&\beta_m&3m+4
 \end{pmatrix},
\end{equation}
where
\begin{equation*}
 \alpha_m=\frac{m^2+6m+4}{2},
 \qquad
 \beta_m=-\frac{m+2}{2}\sqrt m.
\end{equation*}
A direct diagonalization yields
\begin{equation*}
 \operatorname{Spec}(S_m)
 =\left\{2m+2,\frac{m^2+6m+4}{2},
                 \frac{m^2+8m+8}{2}\right\}.
\end{equation*}

Next, for any unit vector $b=(b_j)_{j\in B}$ with
$\sum_{j\in B}b_j=0$, put
\begin{equation*}
 w_B(b)=\frac1{\sqrt m}
 \sum_{\substack{i\in A\\j\in B}}b_j\mathbf e_{W_{ij}}.
\end{equation*}
Then
\begin{equation*}
 A_\infty(x)w_B(b)=(2m+4)w_B(b).
\end{equation*}
Finally, let $Z=(z_{ij})_{i\in A,\,j\in B}$ satisfy
\begin{equation*}
 \sum_{j\in B}z_{ij}=0\quad(i\in A),
 \qquad
 \sum_{i\in A}z_{ij}=0\quad(j\in B).
\end{equation*}
Then
\begin{equation*}
 w(Z)=\sum_{\substack{i\in A\\j\in B}}z_{ij}\mathbf e_{W_{ij}}
 \qquad\text{satisfies}\qquad
 A_\infty(x)w(Z)=\frac{m^2+8m+8}{2}w(Z).
\end{equation*}
Let $\{a^{(\mu)}\}_{\mu=1}^{m-1}$ and
$\{b^{(\nu)}\}_{\nu=1}^{m-1}$ be orthonormal bases of the zero-sum
subspaces
\begin{equation*}
 \mathbb R_0^A
 :=\left\{a=(a_i)_{i\in A}\in\mathbb R^A:
           \sum_{i\in A}a_i=0\right\},
 \qquad
 \mathbb R_0^B
 :=\left\{b=(b_j)_{j\in B}\in\mathbb R^B:
           \sum_{j\in B}b_j=0\right\},
\end{equation*}
respectively. Then $\mathbb{R}^{S_1(x)}$ has the following orthogonal decomposition
\begin{equation*}
  \begin{aligned}
    \mathbb{R}^{S_1(x)}=&\Span\{p_0,q_0,w_0\}\bigoplus_{\mu=1}^{m-1}\Span\{p(a^{(\mu)}),q(a^{(\mu)}),w_A(a^{(\mu)})\}\bigoplus_{\nu=1}^{m-1}\Span\{w_B(b^{(\nu)})\}\\
    &\bigoplus\left\{w((z_{ij})_{i\in A,j\in B}):\ \sum_{j\in B}z_{ij}=0,\ \sum_{i\in A}z_{ij}=0\right\}.
  \end{aligned}
\end{equation*}
 All the displayed eigenvalues
are at least $2$, and \eqref{eq:star-twisted-middle-constant-block} contains
the eigenvalue $2$.
Together with reflection symmetry and
\eqref{eq:star-twisted-endpoint-curvature}, we obtain
\begin{equation}\label{eq:star-twisted-CD}
 \mathcal K_\infty^{G_D^\star}(x)=2
 \qquad\text{for every }x\in V(G_D^\star).
\end{equation}
In particular, $G_D^\star$ satisfies $\CD(2,\infty)$.

We next identify its spectral gap and visible graph.  Every vertex in $L_k$
has $k$ neighbours in $L_{k-1}$ and $D-k$ neighbours in $L_{k+1}$, while all
additional edges remain inside $L_k$.  Hence the function
\begin{equation}\label{eq:star-twisted-rank-function}
 f(k,A)=m-k
\end{equation}
satisfies
\begin{equation*}
 L_{G_D^\star}f(k,A)=k-(D-k)=-2f(k,A).
\end{equation*}
The Lichnerowicz bound and \eqref{eq:star-twisted-CD} now give
$\lambda_1(G_D^\star)=2$.

For distinct vertices of the same layer in $H_D^\star$, the numbers of common
parents and common children are equal and belong to $\{0,1\}$.  Then \Cref{thm:hypercube-cospectrality} gives
\begin{equation}\label{eq:star-twisted-H-gap}
 \lambda_1(H_D^\star)=2.
\end{equation}
Let $K_{\mathrm{hor}}$ be the graph whose vertex set is $V(H_D^\star)$ and whose edge set is the union of additional edges added in
\eqref{eq:star-twisted-horizontal-edges}.  For $g\in E_2(G_D^\star)$,
we have $g\perp\one$, and hence
\begin{align*}
 2\|g\|^2
 &=\langle g,-L_{G_D^\star}g\rangle\\
 &=\langle g,-L_{H_D^\star}g\rangle
   +\langle g,-L_{K_{\mathrm{hor}}}g\rangle
 \geq2\|g\|^2.
\end{align*}
Equality forces
\begin{equation*}
 \langle g,-L_{K_{\mathrm{hor}}}g\rangle
 =\sum_{uv\in E(K_{\mathrm{hor}})}(g(u)-g(v))^2=0.
\end{equation*}
Each subgraph $K_{\mathrm{hor}}[L_k]$ is a connected graph, so $g$ is constant
on every $L_k$; write this value as $c_k$.  Since
$L_{G_D^\star}g=L_{H_D^\star}g=-2g$, the constants satisfy
\begin{equation*}
 (D-2)c_k=kc_{k-1}+(D-k)c_{k+1}.
\end{equation*}
The equation $-2c_0=D(c_1-c_0)$ gives $c_1=\frac{D-2}{D}c_0$, and this recurrence gives
\begin{equation*}
 c_k=\frac{D-2k}{D}c_0.
\end{equation*}
Therefore $E_2(G_D^\star)=\operatorname{span}\{f\}$. All added edges in \eqref{eq:star-twisted-horizontal-edges} are collapsed, whereas $f$ changes by one on every edge
of $H_D^\star$.  This proves $(G_D^\star)_{\vis}=H_D^\star$. 

Finally, suppose $m\geq2$ and fix
$R\in\binom{[D]}{m-1}$.  The two vertices $(m-1,R)$ and $(m+1,R)$ have the
$m+1$ common neighbours
\begin{equation*}
 \{(m,R\cup\{j\}):j\in[D]\setminus R\}
\end{equation*}
in $H_D^\star$.  Two distinct vertices of a hypercube have either zero or two
common neighbours.  Since $m+1\geq3$, it follows that
$H_D^\star\not\cong H_D$.  When $D=2$, the construction still satisfies all
the preceding curvature and spectral conclusions, but
$H_2^\star\cong H_2$.
\end{proof}

\begin{theorem}
\label{thm:non-hypercube-visible-every-dimension}
For every integer $d\geq4$, there exists a connected irregular graph $G$
satisfying $\CD(2,\infty)$ and $\lambda_1(G)=2$ such that $G_{\vis}$ is
$d$-regular and
\begin{equation*}
 G_{\vis}\not\cong H_d.
\end{equation*}
\end{theorem}
\begin{figure}[htbp]
  \centering
  \label{fig:star-twist-cartesian-H1}

\begin{tikzpicture}[
  x=0.50cm,
  y=0.70cm,
  graph edge/.style={draw=black,line width=0.5pt},
  matching edge/.style={draw=black,line width=0.5pt},
  graph vertex/.style={circle,draw=black,fill=white,line width=0.55pt,
    minimum size=4.8pt,inner sep=0pt}
]

\coordinate (zero-p) at (0,0);

\coordinate (zero-a1) at (-4,-1.4);
\coordinate (zero-a2) at (-2,-1.4);
\coordinate (one-p)   at ( 0,-1.4);
\coordinate (zero-a3) at ( 2,-1.4);
\coordinate (zero-a4) at ( 4,-1.4);

\coordinate (zero-e12) at (-4.5,-2.8);
\coordinate (one-a1)   at (-3.5,-2.8);
\coordinate (zero-e13) at (-2.5,-2.8);
\coordinate (one-a2)   at (-1.5,-2.8);
\coordinate (zero-e14) at (-0.5,-2.8);
\coordinate (zero-e23) at ( 0.5,-2.8);
\coordinate (one-a3)   at ( 1.5,-2.8);
\coordinate (zero-e24) at ( 2.5,-2.8);
\coordinate (one-a4)   at ( 3.5,-2.8);
\coordinate (zero-e34) at ( 4.5,-2.8);

\coordinate (one-e12) at (-4.5,-4.2);
\coordinate (zero-b1) at (-3.5,-4.2);
\coordinate (one-e13) at (-2.5,-4.2);
\coordinate (zero-b2) at (-1.5,-4.2);
\coordinate (one-e14) at (-0.5,-4.2);
\coordinate (one-e23) at ( 0.5,-4.2);
\coordinate (zero-b3) at ( 1.5,-4.2);
\coordinate (one-e24) at ( 2.5,-4.2);
\coordinate (zero-b4) at ( 3.5,-4.2);
\coordinate (one-e34) at ( 4.5,-4.2);

\coordinate (one-b1) at (-4,-5.6);
\coordinate (one-b2) at (-2,-5.6);
\coordinate (zero-q) at ( 0,-5.6);
\coordinate (one-b3) at ( 2,-5.6);
\coordinate (one-b4) at ( 4,-5.6);

\coordinate (one-q) at (0,-7);

\foreach \v in {p,a1,a2,a3,a4,e12,e13,e14,e23,e24,e34,b1,b2,b3,b4,q}{
  \draw[matching edge] (zero-\v)--(one-\v);
}

\foreach \copy in {zero,one}{
  \foreach \i in {1,2,3,4}{
    \draw[graph edge] (\copy-p)--(\copy-a\i);
    \draw[graph edge] (\copy-b\i)--(\copy-q);
  }
  \foreach \e/\i/\j in {12/1/2,13/1/3,14/1/4,23/2/3,24/2/4,34/3/4}{
    \draw[graph edge] (\copy-a\i)--(\copy-e\e);
    \draw[graph edge] (\copy-a\j)--(\copy-e\e);
    \draw[graph edge] (\copy-e\e)--(\copy-b\i);
    \draw[graph edge] (\copy-e\e)--(\copy-b\j);
  }
}

\foreach \copy in {zero,one}{
  \foreach \v in {p,a1,a2,a3,a4,e12,e13,e14,e23,e24,e34,b1,b2,b3,b4,q}{
    \node[graph vertex] at (\copy-\v) {};
  }
}

\node at (0,-7.75) {$H_4^\star\square H_1$};

\begin{scope}[xshift=7.7cm]
  \coordinate (h-empty) at (0,0);

  \foreach \i/\x in {1/-4,2/-2,3/0,4/2,5/4}{
    \coordinate (h-s\i) at (\x,-1.4);
  }

  \foreach \e/\x in {12/-4.5,13/-3.5,14/-2.5,15/-1.5,23/-0.5,
                      24/0.5,25/1.5,34/2.5,35/3.5,45/4.5}{
    \coordinate (h-e\e) at (\x,-2.8);
  }

  \foreach \t/\x in {123/-4.5,124/-3.5,125/-2.5,134/-1.5,135/-0.5,
                      145/0.5,234/1.5,235/2.5,245/3.5,345/4.5}{
    \coordinate (h-t\t) at (\x,-4.2);
  }

  \foreach \q/\x in {1234/-4,1235/-2,1245/0,1345/2,2345/4}{
    \coordinate (h-q\q) at (\x,-5.6);
  }

  \coordinate (h-full) at (0,-7);

  \foreach \i in {1,2,3,4,5}{
    \draw[graph edge] (h-empty)--(h-s\i);
  }

  \foreach \e/\i/\j in {12/1/2,13/1/3,14/1/4,15/1/5,23/2/3,
                         24/2/4,25/2/5,34/3/4,35/3/5,45/4/5}{
    \draw[graph edge] (h-s\i)--(h-e\e);
    \draw[graph edge] (h-s\j)--(h-e\e);
  }

  \foreach \t/\a/\b/\c in {
    123/12/13/23,124/12/14/24,125/12/15/25,
    134/13/14/34,135/13/15/35,145/14/15/45,
    234/23/24/34,235/23/25/35,245/24/25/45,345/34/35/45}{
    \draw[graph edge] (h-e\a)--(h-t\t);
    \draw[graph edge] (h-e\b)--(h-t\t);
    \draw[graph edge] (h-e\c)--(h-t\t);
  }

  \foreach \q/\a/\b/\c/\d in {
    1234/123/124/134/234,1235/123/125/135/235,
    1245/124/125/145/245,1345/134/135/145/345,
    2345/234/235/245/345}{
    \draw[graph edge] (h-t\a)--(h-q\q);
    \draw[graph edge] (h-t\b)--(h-q\q);
    \draw[graph edge] (h-t\c)--(h-q\q);
    \draw[graph edge] (h-t\d)--(h-q\q);
  }

  \foreach \q in {1234,1235,1245,1345,2345}{
    \draw[graph edge] (h-q\q)--(h-full);
  }

  \node[graph vertex] at (h-empty) {};
  \node[graph vertex] at (h-full) {};
  \foreach \i in {1,2,3,4,5}{
    \node[graph vertex] at (h-s\i) {};
  }
  \foreach \e in {12,13,14,15,23,24,25,34,35,45}{
    \node[graph vertex] at (h-e\e) {};
  }
  \foreach \t in {123,124,125,134,135,145,234,235,245,345}{
    \node[graph vertex] at (h-t\t) {};
  }
  \foreach \q in {1234,1235,1245,1345,2345}{
    \node[graph vertex] at (h-q\q) {};
  }

  \node at (0,-7.75) {$5$-dimensional hypercube $H_5
  $};
\end{scope}
\end{tikzpicture}
\end{figure}
\begin{proof}
If $d$ is even, take $G=G_d^\star$.  All the assertions follow from
\Cref{prop:even-dimensional-star-twist}; the graph is irregular by
\eqref{eq:star-twisted-degrees}.

Now suppose that $d$ is odd.  Write $d=D+1$, where $D=2m\geq4$, and set
\begin{equation*}
 \widetilde G=G_D^\star\square H_1.
\end{equation*}
We first verify the curvature condition.  For a function $u$ on the Cartesian
product, direct expansion of the $\Gamma_2^{\widetilde G}$ gives
\begin{align*}
 \Gamma_2^{\widetilde G}(u)(x,\varepsilon)
 ={}&\Gamma_2^{G_D^\star}\bigl(u(\,\cdot\,,\varepsilon)\bigr)(x)
   +\Gamma_2^{H_1}\bigl(u(x,\,\cdot\,)\bigr)(\varepsilon)\\
 &+\frac12\sum_{y\sim x}\sum_{\eta\sim\varepsilon}
 \bigl(u(y,\eta)-u(y,\varepsilon)-u(x,\eta)
       +u(x,\varepsilon)\bigr)^2.
\end{align*}
The last term is nonnegative.  Since both $G_D^\star$ and $H_1$ satisfy
$\CD(2,\infty)$, while the corresponding first-order forms add on a
Cartesian product, it follows that $\widetilde G$ satisfies
$\CD(2,\infty)$.

The Laplacian on the product is
\begin{equation*}
 -L_{\widetilde G}
 =(-L_{G_D^\star})\otimes I+I\otimes(-L_{H_1}).
\end{equation*}
By \Cref{prop:even-dimensional-star-twist},
$\lambda_1(G_D^\star)=2$ and $E_2(G_D^\star)=\Span\{f\}$, where
$f(k,A)=m-k$.  If $h$ is the function on $H_1$ given by
$h(0)=1$ and $h(1)=-1$, then $E_2(H_1)=\Span\{h\}$.  Hence
\begin{equation}\label{eq:odd-star-product-two-eigenspace}
 \lambda_1(\widetilde G)=2,
 \qquad
 E_2(\widetilde G)
 =\Span\bigl\{(x,\varepsilon)\mapsto f(x),
               (x,\varepsilon)\mapsto h(\varepsilon)\bigr\}.
\end{equation}
Thus $\widetilde G$ is Lichnerowicz-sharp.  Moreover,
\begin{equation*}
 \deg_{\widetilde G}((k,A),\varepsilon)
 =D+k(D-k)+1,
\end{equation*}
so $\widetilde G$ is irregular.

Equation \eqref{eq:odd-star-product-two-eigenspace} also determines every
visible edge.  An edge in the $G_D^\star$-direction is visible precisely when
its projection is an edge of $H_D^\star$, and every edge in the $H_1$-direction
is visible.  Therefore
\begin{equation*}
 (\widetilde G)_{\vis}=H_D^\star\square H_1,
\end{equation*}
which is $(D+1)$-regular.

It remains to show that this visible graph is not a hypercube.  Choose
$R\in\binom{[D]}{m-1}$ and put
\begin{equation*}
 u=(m-1,R),\qquad v=(m+1,R).
\end{equation*}
The vertices $(u,0)$ and $(v,0)$ in $H_D^\star\square H_1$ have exactly the
$m+1$ common neighbours
\begin{equation*}
 \bigl\{((m,R\cup\{j\}),0):j\in[D]\setminus R\bigr\}.
\end{equation*}
Since $m+1\geq3$, whereas two distinct vertices of a hypercube have either
zero or two common neighbours, we obtain
$H_D^\star\square H_1\not\cong H_{D+1}=H_d$.
\end{proof}
\begin{remark}
    Here $d\geq 4$ is necessary. Since the only candidates satisfying structure described in \Cref{thm:intrinsic-Gamma} are hypercubes when $d\le 3$.
\end{remark}

\section{Acknowledgments}
This work was supported by the Scientific Research Innovation Capability Support Project for Young Faculty (Grant No. SRICSPYF-ZY2025160) and the National Natural Science Foundation of China (Grant No. 12431004).

\bibliographystyle{plainurl}
\bibliography{rigid}

\end{document}